\documentclass[a4paper,10pt]{siamltex1213}
\usepackage{graphicx}
\usepackage{amsfonts}
\usepackage{amsmath}
\usepackage{amssymb}
\usepackage{fancyhdr}
\usepackage{mathtools}
\usepackage{indentfirst}
\usepackage{algorithm}
\usepackage[noend]{algpseudocode}
\usepackage{placeins}
\usepackage{listings}
\usepackage{caption}
\usepackage{subfigure}
\usepackage{lipsum}
\usepackage{xcolor}
\usepackage{comment}
\usepackage{url}
\usepackage{bm}
\usepackage{multirow}

\renewcommand{\figurename}{Figure}
\def\longrightharpoonup{\relbar\joinrel\rightharpoonup}
\def\longleftharpoondown{\leftharpoondown\joinrel\relbar}

\def\longrightleftharpoons{
  \mathop{
    \vcenter{
       \hbox{
       \ooalign{
          \raise1pt\hbox{$\longrightharpoonup\joinrel$}\crcr
  	  \lower1pt\hbox{$\longleftharpoondown\joinrel$}
	}
      }
    }
  }
}

\newtheorem{remark}{Remark}

\newtheorem{assump}{Assumption}
\newtheorem{prop}{Proposition}

\newcommand\blfootnote[1]{%
  \begingroup
  \renewcommand\thefootnote{}\footnote{#1}%
  \addtocounter{footnote}{-1}%
  \endgroup
}
\begin{document}
\title{Statistical analysis of the first passage path ensemble of jump processes}
\date{}
\author{
  Max von Kleist\,$^1$ \and
Christof Sch\"utte\,$^{1,\,2}$ \and
  Wei Zhang\,$^1$ }
\blfootnote{$^1$\,Institute of Mathematics, Freie Universit\"{a}t Berlin, Arnimallee 6, 14195 Berlin, Germany}
\blfootnote{$^2$\,Zuse Institute Berlin, Takustrasse 7, 14195 Berlin, Germany}
\blfootnote{Email : vkleist@zedat.fu-berlin.de, schuette@zib.de, wei.zhang@fu-berlin.de}

\maketitle

\begin{abstract}
The transition mechanism of jump processes between two different subsets in state space 
reveals important dynamical information of the processes and therefore has attracted
considerable attention in the past years. In this paper, we study the first passage path ensemble 
of both discrete-time and continuous-time jump processes on a finite state space.
The main approach is to divide each first passage path into nonreactive and
reactive segments and to study them separately. The analysis can be applied to jump
processes which are non-ergodic, as well as continuous-time jump processes where the waiting time
distributions are non-exponential. In the particular case that the jump
processes are both Markovian and ergodic, our analysis elucidates the
relations between the study of the \textit{first passage paths} and the
study of the \textit{transition paths} in transition path theory. We provide algorithms to numerically compute statistics
  of the first passage path ensemble. The computational complexity of these algorithms
  scales with the complexity of solving a linear system, for which efficient methods are available. Several examples
  demonstrate the wide applicability of the derived results across research areas.
\end{abstract}
\begin{keywords}
jump process, non-ergodic process, non-exponential distribution, first passage path, transition path theory
\end{keywords}

\section{Introduction}
\label{sec-intro}
(Markov) jump process has been extensively studied in the past decades
and nowadays it becomes a standard mathematical model that is widely applied to problems arising
from physics, chemistry, biology, etc~\cite{durrett_prob_example,norris-mc}. Among many important
topics in the study of jump processes, the \textit{first passage paths} have attracted
considerable attention within different disciplines, where the main purpose
is to understand the transition mechanism of the system between
different subsets in state space and to e.e. access how much time the transitions typically take~\cite{metzler2014first}. Examples include reaction networks in chemistry \cite{fptp_chemical,Kleist2012}, phenotypic switches in cell
biology~\cite{Yousef2015}, conformational changes in molecular dynamics~\cite{milestoning_network2013, reactionPath_mfpt}, disease spreading within certain geographical
areas~\cite{Balcan2009,first-passage-time-nature}, as well as spread of information
on social networks~\cite{prl_fpt_heterogeneous}. In these contexts, studies of the first passage
paths are often very helpful to understand the underlying processes and to
foster- or prevent the transition events. 

A common situation that one often encounters in the study of many real-world
applications is that the system exhibits metastability
and transition events become very rare~\cite{msm_generation, huisinga2004,schuette2013metastability}.
To study the transition events in these scenarios, transition path theory
(TPT) has been developed both for diffusion processes on continuous state
space~\cite{tpt2010,towards_tpt2006,tpt_eric2006} and for Markov jump
processes on discrete state space~\cite{tpt_jump, Cameron2014}.
It provides a probabilistic framework to analyze the statistical properties of system's reactive trajectories 
(following the terminology in \cite{Lu2014}) and enables us to answer the
questions which are important in order to understand the transition mechanism
of the system. In the discrete state space setting, TPT can be applied to
study Markov jump processes which are ergodic with a unique invariant measure~\cite{tpt_jump, Vanden-Eijnden2014}.

Meanwhile, inspired by the wide applications that have emerged due to the rapid development of network science,
scientific interest has been extended to study processes which go beyond
ergodic Markov jump processes~\cite{Poletto2013,Scholtes2014,Rosvall2014}. 
One simple example when ergodicity is violated is the random walk on a directed graph which contains sink-states (states with no outward
edges). Another example of a non-Markovian process is the continuous-time random walk with possibly
 non-exponential waiting time distributions. These kind of processes have been applied to model the burst and memory effects in real-world networks~\cite{non_poisson_jump,path_mem_carse_graining,burst_mem}.
The above described processes are no longer ergodic Markov jump processes and therefore TPT can not be applied directly. However, interesting insights can be obtained when
studying the transition mechanism from one subset to another, to determine how much time the transitions typically take and to
identify key nodes or edges for the transition events.
To our best knowledge, these questions have not been systematically addressed
in the non-Markovian and non-ergodic setting (see related study in \cite{path_mem_carse_graining}).

In the current work we consider the first passage path
ensemble of both continuous-time and discrete-time jump processes on a
discrete, finite state space.
While we are strongly influenced by TPT and will study similar quantities such as the probability of visiting
each node and probability fluxes on each edge, we emphasize that both the
subject being studied and the setting are
different from the study of TPT~\cite{tpt_jump, Cameron2014}.
Specifically, we will study the first passage path by dividing it into
nonreactive and reactive segments.
While in the ergodic case the reactive segments coincide with the reactive
trajectories in TPT, in this work we also study the statistics of nonreactive
segments and their relations with the statistics of the entire \textit{first passage paths}.
Furthermore, in our study the processes are allowed to be non-ergodic, and in the continuous-time scenario the waiting time on each node can be also non-exponentially distributed.
The main contributions of this paper can be summarized as follows. Firstly,
statistical properties of both nonreactive and reactive path ensembles have been
analyzed, which enables us to compute several important quantities
associated to the reactive- and nonreactive ensembles. Their relations to the entire first passage
path ensemble are obtained subsequently. Secondly, since the processes are not necessarily ergodic or in stationary anymore, 
our analysis (of the reactive segments) can be viewed as an extension of TPT to nonequilibrium path ensembles. 
When further assuming that the processes are ergodic Markov jump processes and the first passage path
ensemble is in equilibrium, our analysis recovers TPT and allows to
elucidate the statistical connection between transition paths in TPT and the first passage paths.
Thirdly, the numerical implementation of the derived results is discussed, which is useful in many applications.

The paper is organized as follows:
The first passage path ensemble of discrete-time Markov jump processes
is studied in Section~\ref{sec-reactive}. The first passage path ensemble 
of continuous-time jump processes with general waiting time distributions is
considered in Section~\ref{sec-ctjp}.
Connections of our analysis with TPT are discussed in Section~\ref{sec-ergodic}
where we recover the results of TPT when the processes are ergodic and the path
ensemble is in equilibrium.
Algorithmic issues are discussed in Section~\ref{sec-algo}
and several numerical examples are presented in
Section~\ref{sec-example} to illustrate our analysis framework.
Conclusions and further discussions are present in Section~\ref{sec-conclusion}.
Finally, some supplementary analysis related to Section~\ref{sec-reactive} and
Section~\ref{sec-ctjp} is provided in Appendix~\ref{appsec-1} and Appendix~\ref{appsec-2}.

For the reader's convenience, important notations which will be used in the current work are summarized in Table~\ref{tab-notation}.

\begin{table}[htbp]\caption{Notations used throughout the paper}
  \small
\begin{center}
  \begin{tabular}{l@{\hskip 0.2cm}p{5.7cm}@{\hskip 0.5cm}|l@{\hskip 0.2cm}p{5.7cm}}
\hline
    $\mathcal{N}_x$ & set of neighbor nodes & $\mathcal{T}$ &  set of sink nodes \\
    $V^{-}$ &  subset of node set $V$ & $V^{+}$ &  subset of node set $V$\\
    $q$ & committor function & $q^-$ & backward committor function \\
    $m$ & invariant measure of discrete process & $Z$ & normalization constant \\
    $\Xi_{non}$ & nonreactive trajectory ensemble & $\Xi_{r}$ & reactive trajectory ensemble \\
    $\sigma_x$ & last hitting time of set $A$ & $\mu$ & initial distribution on set $A$ \\
    $\mu_r$ & initial distribution of reactive trajectories & $p$ & transition probability of discrete process \\
$\psi$ & probability density of the waiting time along an edge & $\bar{p}$ & transition probability of the nonreactive trajectories \\
    $\theta$ & average number of times that the first passage paths visit a node & $\widetilde{p}$ & transition probability of the reactive trajectories \\
$\bar{\theta}'$ & average number of times that nonreactive trajectories
(except the last node) visit a node & $J$ & average number of times that
the first passage paths visit an edge \\
$\bar{\theta}$ & average number of times that nonreactive trajectories visit a node & $\bar{J}$ & average number of times that nonreactive
trajectories visit an edge \\
$\widetilde{\theta}$ & average number of times that reactive trajectories visit a node &
$\widetilde{J}$ & average number of times that the reactive trajectories visit an edge \\
  $a$ & probability that system stay at a node longer than certain amount of time &
$T$ & average total time of the first passage paths \\
$b$ & probability density that the system jumps along an edge at a certain time
     & $\bar{T}$ & average total time of the nonreactive trajectories \\
$\kappa$ & average amount of time that the system stays at a node
    & $\widetilde{T}$ & average total time of the reactive trajectories \\
\hline
\end{tabular}
\end{center}
\label{tab-notation}
\end{table}
\section{Analysis of the first passage path ensemble : discrete-time Markov jump processes}
    \label{sec-reactive}
    In this section, we study the first passage path ensemble when the system is a discrete-time Markov jump process.
    After introducing various useful notations in
    Subsection~\ref{subsec-prepare}, we will analyze the statistics of the nonreactive
    ensemble and the reactive ensemble in Subsection~\ref{subsec-nonreactive} and
    Subsection~\ref{subsec-reactive}, respectively. Their connections with the
    whole first passage path ensemble are discussed in Subsection~\ref{subsec-fpp}.

    We also emphasize that although several quantities that we will consider are strongly influenced by TPT, the main difference is
    that they are related to the first passage path ensemble which may be out
    of equilibrium and therefore do not reply on the existence of invariant
    measure. Connections of our analysis with TPT will be further discussed in Section~\ref{sec-ergodic}. 
    \subsection{Preparations}
    \label{subsec-prepare}
    We start by introducing the processes we will consider and the notations that
    will be used in this paper.
Let $G=(V, E)$ be the graph representation of a directed network $G$, where $V$ is
the set of nodes and $E \subseteq V \times V$ is the set of edges.
We assume that $V$ is a finite set and, without loss of generality, $V =
\{1,2,\cdots, n\}$ for some $n>1$.
We will write $x\rightarrow y$ if $(x,y) \in E$ and denote $\mathcal{N}_x = \{\,y \in V~|~x \rightarrow y\}$
as the set consisting of all nodes which can be directly reached from node $x$.
For simplicity, we assume that there are no loop edges in $G$, i.e. $x
\not\in \mathcal{N}_x$, for $\forall x \in V$.
  Since we are also interested in the case when the network $G$ contains sinks, we
  allow the case when $\mathcal{N}_x = \emptyset$ for some node $x \in V$ and
  denote $\mathcal{T} = \{x \in V~|~ \mathcal{N}_x = \emptyset\}$ to be the set of sink
  nodes.

 Suppose that two disjoint nonempty subsets $A, B \subset V$ are
    given. And let $\mu$ be a probability distribution on set $A$. Consider the
    discrete-time Markov process defined by the jump probability
    $p(y\,|\,x)$ for $x, y \in V$, which is the probability that the system
    will jump to state $y$ if its current state is $x$.
Given node $x \in V$, we define two stopping times related to sets $A,B$
\begin{align}
  \tau_{A,x} = \min_{k\ge 0}\big\{\,k ~\big|~ x_0 = x\,, x_k \in A \big\}\,, \quad
  \tau_{B,x} = \min_{k \ge 0}\big\{\,k~\big|~ x_0 = x\,, x_k \in B \big\}\,,
\end{align}
and set $\tau_{A,x}=+\infty$ (or $\tau_{B,x}=+\infty$) if the process will never
reach set $A$ (or $B$) from $x$.
Especially, we have $\tau_{A,x} \equiv 0$ for $x \in A$ and $\tau_{B,x}\equiv 0$ for $x \in B$.
    Assume the process starts from some node $x \in B^c$
    and consider the path until it reaches set $B$.
    Denote the path as $(x_0, x_1, \cdots, x_k)$, then we have $x_0 =
    x$, $x_k \in B$, where $k = \tau_{B,x}$ and $x_l \not \in B$ for $0 \le l < k$.
    Such a path is called the \emph{first passage path}~\cite{first-passage-time-nature,metzler2014first}.
    We also define the \emph{last} hitting time of set $A$ as
    \begin{align}
      \sigma_x= \max_{l \ge 0}\Big\{l ~\Big|~ x_l \in A,\, 0 \le l < \tau_{B,x}\Big\}\,,
	\label{sigma-def}
    \end{align}
    and set $\sigma_x = +\infty$, if $x_l \not\in A$ for $0 \le l < \tau_{B,x}$.
In the following, we will use the notations $\tau_A$, $\tau_B$ and $\sigma$
for simplicity if we
do not explicitly emphasize the initial state $x$.

    Given any first passage path starting from a state in set $A$ (therefore $\sigma <
    +\infty$), we can split this path into two
segments using the last hitting time $0 \le \sigma < \tau_B$.
The first segment $(x_0, x_1, \cdots, x_{\sigma})$, which will be termed as
\emph{nonreactive trajectory} or \emph{nonreactive segment}, consists of the consecutive nodes visited by the
process before it eventually leaves set $A$.
The second segment $(x_{\sigma}, x_{\sigma+1}, \cdots, x_{\tau_B})$, which is called \emph{reactive
trajectory} in transition path theory~\cite{tpt_jump}, is the transition pathway of the process from set $A$
to set $B$ (see Figure~\ref{graph-1}(a) for illustration). We will call it either
\emph{reactive trajectory} or \emph{reactive segment}.
    Equivalently, the reactive trajectory is the segment of the first
    passage path starting from some node in $A$ which is hit by the process for the last time.
    Clearly, we have $x_{\sigma} \in A$ and $x_l \in (A\cup B)^c$ for $\sigma < l < \tau_B$.
    The ensembles of the nonreactive trajectories and reactive trajectories
    are denoted by
    \begin{align}
      \Xi_{non} = \Big\{(x_0, x_1, \cdots, x_\sigma)\Big\}\,,\qquad \Xi_r =
      \Big\{(x_{\sigma}, x_{\sigma+1}, \cdots, x_{\tau_B})\Big\}\,,
      \label{ensembles}
    \end{align}
    respectively, where $(x_0, x_1, \cdots, x_{\tau_B})$ goes over all first
    passage paths with initial state $x_0 \in A$ and $x_0 \sim \mu$.
\begin{figure}[htp]
\centering
\begin{tabular}{l@{\hskip 1.5cm}r}
\subfigure[]{\includegraphics[width=0.53\textwidth]{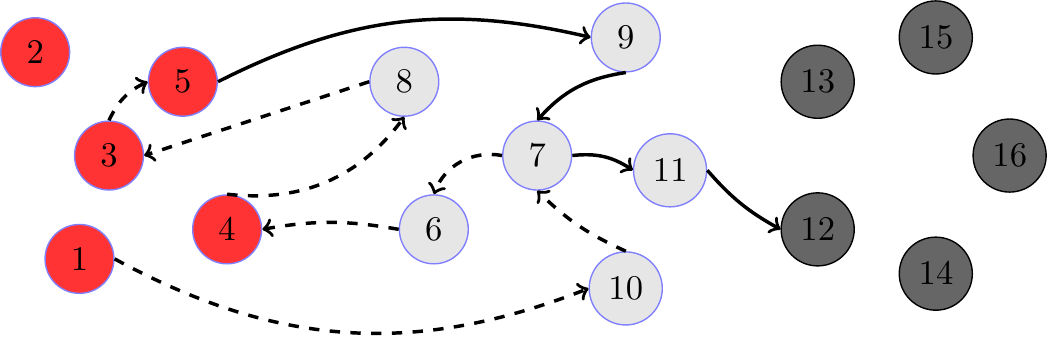}} &
\subfigure[]{\includegraphics[width=0.33\textwidth]{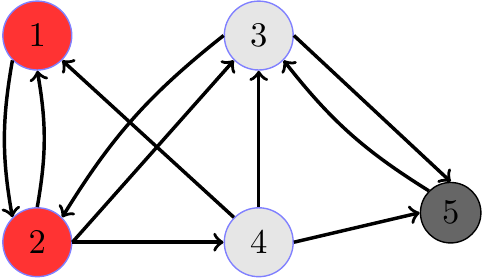}}
\end{tabular}
\caption{(a) Illustration of the nonreactive and reactive trajectories (segments) within a first
passage path starting from node $1$. We have sets $A = \{1,2,3,4,5\}$, $B=\{12,
13,14,15, 16\}$. Dashed line and solid line are the nonreactive trajectory and
reactive trajectory of the first passage path, respectively. In this case, we
have $\sigma = 7$, $x_\sigma = 5$ and $\tau_B = 11$. (b) A simple example
shows the difference between probability distributions $\mu$ and $\mu_r$. We
have sets $A=\{1,2\}$ and $B=\{5\}$. Since the first passage path can only leave set $A$
from node $2$, we have $\mu_r(1) = 0$, $\mu_r(2) = 1.0$ for any initial
probability distribution $\mu$.}
\label{graph-1}
\end{figure}

Given a node $x \in V$, we say that set $A$ (or $B$) is \emph{reachable} from $x$, if there is a path of
finite length $x_0, x_1, \cdots, x_k$, where $k \ge 0$, such that $x_0 = x$, $x_k \in A$ (or
$x_k \in B$), $x_l \in (A \cup B)^c$, $1 \le l \le k-1$, and $p(x_{l+1}\,|\,x_l) > 0$, $0 \le l \le k-1$.
Throughout the paper, we will make the following assumption.
    \begin{assump}
      $(A\cup B)^c \not=\emptyset$. For each node $x \in (A\cup B)^c$,
      either set $A$ or set $B$ is reachable from $x$. Furthermore, set $B$ is
      reachable from each node $x \in A$.
    \label{assump-1}
    \end{assump}

    Under the above assumption, we can conclude by contradiction that $\tau_{A,x} \wedge
    \tau_{B,x} < +\infty$ with probability $1$ for all $x \in V$.
  Therefore, the committor function,
  \begin{align}
  q(x) = \mathbb{P}(\tau_{B,x} < \tau_{A,x})\,, \quad \forall x \in V,
  \label{committor-q}
\end{align}
  corresponding to sets $A, B$ is well defined and will play an important role
  in the following study. It is known that it
   satisfies the equations~\cite{tpt_jump}
    \begin{align}
      \begin{split}
      &\sum_{y \in V} p(y\,|\,x)\, q(y) = q(x)\, , \quad \forall x \in (A\cup B)^c\,, \\
      &q|_A=0, \quad q|_B=1\,.
    \end{split}
    \label{committor}
    \end{align}
We also define two subsets
  \begin{align}
V^-=\Big\{x \,\Big|\, q(x) < 1,\, x \in V\Big\}, \qquad V^+=\Big\{\,x ~\Big|~ \sum_{z \in V} p(z\,|\,x) q(z) > 0\,, x \in B^c\Big\}\,.
    \label{v-minus-plus}
  \end{align}
  Clearly, we have $A \subseteq V^- \subseteq B^c$ and
$A \subseteq V^+ \subseteq B^c$, where the latter is implied by Assumption~\ref{assump-1} together with the following result.
\begin{prop}
  Let $\sigma_x$ be the last hitting time defined in (\ref{sigma-def}) and subsets
  $V^-, V^+$ be defined in (\ref{v-minus-plus}). We have
  \begin{enumerate}
    \item
    $\mathbb{P}(\sigma_x < +\infty) = 1 - q(x)$, $\forall x \in V$.
    \item
      $x \in V^-$ iff $x \in B^c$ and set $A$ is reachable from node $x$.
    \item
      $x \in V^+$ iff $x \in B^c$ and set $B$ is reachable from node $x$.
\end{enumerate}
\label{prop-0}
  \end{prop}
  \begin{proof}
    We will only prove the first two conclusions since the third one can be
    obtained using a similar argument as the second one.
    \begin{enumerate}
      \item
	By definition of $\sigma_x$, we know that the two events
	  $\sigma_x < +\infty$ and $\tau_{A, x} < \tau_{B,x}$ are equivalent.
	  It follows from the definition of $q$ in
	  (\ref{committor-q}) that $\mathbb{P}(\sigma_x < +\infty) =
	  \mathbb{P}(\tau_{A,x} < \tau_{B,x})  = 1 - q(x)$.
	\item
	  Suppose $x \in V^-$. Using the definition of $q$, we know
	  $\tau_{A,x}< +\infty$ with a positive probability, which implies
	  that set $A$ is reachable from $x$. Conversely, suppose $x \in B^c$ and set $A$ is reachable
	  along the path $x_0 = x, x_1, \cdots, x_k \in A$.
	  If $x \not\in V^-$, then $q(x) = 1$ and $x \in (A \cup
	  B)^c$. Applying equation (\ref{committor}), we obtain $q(x_1) = 1$.
	  As $x_l \in (A\cup B)^c$ for $1 \le l \le k-1$, we can repeat the
	  argument and obtain $q(x_k) = q(x_{k-1}) = \cdots = q(x_1) = 1$.
	  However, this contradicts the fact that $q|_A=0$ since $x_k\in A$.
	  Therefore we have $x \in V^-$.
	\end{enumerate}
  \end{proof}
  From Proposition~\ref{prop-0}, we know that Assumption~\ref{assump-1}
  implies $V^- \cup V^+ = B^c$ and $\mathcal{T} \subset B$.
    \subsection{Nonreactive ensemble}
    \label{subsec-nonreactive}
    In this subsection, we study the nonreactive ensemble $\Xi_{non}$ defined in (\ref{ensembles}).
Given a nonreactive trajectory $x_0, x_1, \cdots, x_{\sigma}$, where $x_0 \in
A$ and $\sigma$
is the last hitting time defined in (\ref{sigma-def}), it is clear
that $x_l \in V^-$ for $0\le l \le \sigma$ (Proposition~\ref{prop-0}). Our aim is to define a Markov jump
process whose path ensemble coincides with $\Xi_{non}$. For this purpose, we first
introduce a virtual node $0$ to mark the end of the path and consider the
extended node set $V^-\cup
\{0\}$. A jump from some node $x \in A$ to node $0$ indicates that
$x$ is the last node of the nonreactive trajectory. We also define the transition probability from
node $y \in V^{-}$ to $x \in V^{-} \cup \{0\}$ as
\begin{align}
  \bar{p}(x\,|\,y) = \left\{
    \begin{array}{cl}
      \frac{p(x\,|\,y) (1-q(x))}{1 - q(y)}\,,& \mbox{if}~ x \neq 0 \,, \\
      \sum\limits_{z \in V} p(z\,|\,y)\, q(z) \,, & \mbox{if}~x = 0\,,~ y \in A\,, \\
      0  \,, & \mbox{if}~x = 0\,,~ y \not\in A\,, \\
    \end{array}
    \right.
    \label{p-bar}
\end{align}
and $\bar{p}(x\,|\,0) = \delta(x)$. Using (\ref{committor}), we can verify that $\bar{p}$ is a
probability matrix on set $V^-\cup \{0\}$ with row sum one. Furthermore, we have
\begin{prop}
  The ensemble $\Xi_{non}$ can be generated by trajectories (without the end node $0$) of the Markov jump process on space $V^-\cup \{0\}$ with
  transition probability matrix $\bar{p}$ in (\ref{p-bar}) and initial
  distribution $\mu$.
  \label{prop-1}
\end{prop}
\begin{proof}
  Let $(x_0, x_1, \cdots, x_k, \cdots, x_\sigma) \in \Xi_{non}$ be one
  nonreactive trajectory such that $x_k=y$. For $y \in V^-\cap A^c$ and $x \in
  V^-$, using the
  Markov property and the definition of the committor function $q$, we can compute
  \begin{align*}
     &\mathbb{P}\big(x_{k+1} = x\,\big|\, (x_k=y, \cdots, x_0), \mbox{being
  nonreactive}\big) \\
  = & \frac{\mathbb{P}\big(x_{k+1} = x, x_k=y , \mbox{being
  nonreactive}\big)}{\mathbb{P}\big(x_k=y, \mbox{being
  nonreactive}\big)}\\
  = & \frac{\mathbb{P}\big(x_{k+1} = x\,, x_k=y\big)
\mathbb{P}\big(\mbox{being nonreactive}\,|\, x_{k+1} = x\big)}{\mathbb{P}\big(x_k=y\big)\mathbb{P}\big( \mbox{being nonreactive}\,|\,x_k=y\big)}\\
= & \frac{p(x\,|\,y) (1 - q(x))}{1 - q(y)}\,= \bar{p}(x\,|\,y)\,.
  \end{align*}
  For $y \in A$, event $\{k= \sigma\}$ is equivalent to the event that
  the original process (on $V$) hits set $B$ first (before it returns to $A$) after it leaves set
$A$ from node $y$. The probability of the latter is equal to $\sum\limits_{z \in V}
p(z\,|\,y) q(z)$ and therefore coincides with the probability that the
process (on $V^- \cup \{0\}$) jumps from $y$ to $0$.
The case for $x \in V^-, y\in A$ can be argued using a similar argument.
\end{proof}

Given $x \in V^-$, let $\bar{\theta}(x)$ be the average number of times that node $x$
is visited by the nonreactive trajectories, i.e.
\begin{align}
  \bar{\theta}(x) = \langle \sum_{l=0}^{\sigma} \mathbf{1}_{x}(x_l)\rangle \,, \quad x
  \in V^-\,,
  \label{theta-bar}
\end{align}
where $\mathbf{1}_{x}(\cdot)$ is the indicator function, $\langle\cdot\rangle$ denotes the ensemble average of $\Xi_{non}$, or
equivalently, the ensemble average of the first passage paths with the initial distribution $\mu$.  It
is straightforward to verify
\begin{align}
  \sum_{x \in V^-} \bar{\theta}(x) = \langle \sum_{l=0}^{\sigma}\sum_{x \in V^-}
  \mathbf{1}_{x}(x_l)\rangle = \langle \sigma \rangle + 1\,.
  \label{theta-bar-sum}
\end{align}
Furthermore, we have
\begin{prop}
  For $x \in V^-$, function $\bar{\theta}(\cdot)$ defined in (\ref{theta-bar})
  satisfies the equation
\begin{align}
  \bar{\theta}(x) = \mu(x)\mathbf{1}_{A}(x) + \sum_{y \in V^-} \bar{\theta}(y)
  \bar{p}(x\,|\,y)\,,
    \label{theta-bar-eqn}
\end{align}
where probabilities $\bar{p}$ are given in (\ref{p-bar}).
\label{prop-2}
\end{prop}
\begin{proof}
  Using Proposition~\ref{prop-1}, we can rewrite the
  transition probabilities $\bar{p}$ using the ensemble average, i.e.
  \begin{align}
    \bar{p}(x\,|\,y) = \frac{\langle \sum\limits_{l=0}^\sigma \mathbf{1}_y(x_l)
    \mathbf{1}_x(x_{l+1})\rangle}{\langle \sum\limits_{l=0}^\sigma \mathbf{1}_y(x_{l})\rangle}
    \label{p-bar-path}
  \end{align}
  where $x,y \in V^-$ and we have set $x_{\sigma+1} = 0$, which marks the end of the nonreactive trajectory.
  Therefore, using the definition of $\bar{\theta}$ in (\ref{theta-bar}) and summing up $y \in V^-$, we obtain
  \begin{align*}
    \sum_{y \in V^-} \bar{\theta}(y) \bar{p}(x\,|\,y) = \sum_{y \in V^-} \langle \sum\limits_{l=0}^\sigma \mathbf{1}_y(x_l) \mathbf{1}_x(x_{l+1})\rangle
    =\langle \sum_{l=1}^{\sigma+1} \mathbf{1}_x(x_l)\rangle \,.
\end{align*}
Since $x \in V^-$, we know $\mathbf{1}_x(x_{\sigma + 1}) = 0$ and therefore
\begin{align}
  \sum_{y \in V^-} \bar{\theta}(y) \bar{p}(x\,|\,y) = \langle \sum_{l=0}^{\sigma}
  \mathbf{1}_x(x_l)\rangle - \langle \mathbf{1}_x(x_0)\rangle \, =
  \bar{\theta}(x) - \mu(x) \mathbf{1}_A(x)\,,
\end{align}
where we have used the definition of $\bar{\theta}$ again and the fact that $x_0 \sim \mu$.
\end{proof}

Especially, summing up $x \in V^-$
in (\ref{theta-bar-eqn}) and using the fact that $\bar{p}$ has row
sum one, we can verify the equality
\begin{align}
  \sum_{x \in A} \bar{\theta}(x) \Big[\sum_{z \in V} p(z\,|\,x) q(z)\Big] = \sum_{x
  \in A} \bar{\theta}(x) \bar{p}(0\,|\,x) = 1\,.
  \label{theta-bar-sum-1}
\end{align}

It is also useful to introduce
\begin{align}
  \mu_r(x) = \langle \mathbf{1}_x(x_\sigma)\rangle\,, \quad x\in A\,,
  \label{mu-r}
\end{align}
which is the probability distribution of the last hitting node on set $A$ and
satisfies $\sum\limits_{x \in A} \mu_r(x) = 1$ (see Figure~\ref{graph-1} (b) for illustration).
Furthermore, in analogy to (\ref{theta-bar}), we define
\begin{align}
  \bar{\theta}'(x) = \langle \sum_{l=0}^{\sigma-1} \mathbf{1}_x(x_l)\rangle =
  \bar{\theta}(x) - \mu_r(x)\mathbf{1}_A(x)\,,
  \quad x\in V^-\,,
  \label{theta-bar-prime}
\end{align}
which coincides with $\bar{\theta}$ on $A^c$ and the summation above is interpreted to be zero when $\sigma = 0$.
The following relations are straightforward.
    \begin{prop}
      Let $\bar{\theta}$ and the probabilities $\bar{p}$ be defined in (\ref{theta-bar}) and
      (\ref{p-bar}), respectively. Then for $x \in A$,
      \begin{align}
	\begin{split}
	  \mu_r(x) =& \bar{\theta}(x) \bar{p}(0\,|\,x) = \bar{\theta}(x)
	\Big[\sum_{z \in V} p(z\,|\,x) q(z)\Big]\,,\\
	\bar{\theta}'(x) =& \bar{\theta}(x) - \mu_r(x) = \bar{\theta}(x)\sum_{z
      \in V^-} \bar{p}(z\,|\,x)\,.
      \end{split}
	\label{mu-r-theta-bar}
      \end{align}
      \label{prop-3}
    \end{prop}
    \begin{proof}
      We only need to prove the first equality in (\ref{mu-r-theta-bar}).
      Notice that (\ref{mu-r}) can be written as
      \begin{align*}
	\mu_r(x) = \langle \mathbf{1}_x(x_\sigma)\rangle  = \langle
	\sum_{l=0}^{\sigma} \mathbf{1}_x(x_l)
	\mathbf{1}_{0}(x_{l+1})\rangle\,,
      \end{align*}
      where we have used the convention that $x_{\sigma + 1} = 0$. Similarly
      as in (\ref{p-bar-path}), applying Proposition~\ref{prop-1}, we have
      \begin{align*}
	\bar{p}(0\,|\,x) = \frac{\langle
	\sum\limits_{l=0}^{\sigma} \mathbf{1}_x(x_l)
      \mathbf{1}_{0}(x_{l+1})\rangle}{
\langle \sum\limits_{l=0}^{\sigma} \mathbf{1}_x(x_l)\rangle}
= \frac{\mu_r(x)}{\bar{\theta}(x)}\,, \quad x \in A\,.
      \end{align*}
      Therefore, we conclude that
      \begin{align*}
	\mu_r(x) = \bar{\theta}(x) \bar{p}(0\,|\,x) = \bar{\theta}(x) \Big[\sum_{z \in V} p(z\,|\,x) q(z)\Big]\,.
      \end{align*}
    \end{proof}

We also introduce the \emph{nonreactive probability flux}, which is defined as
\begin{align}
  \bar{J}(x\rightarrow y) = \bar{\theta}(x) \bar{p}(y\,|\,x)\,, \qquad x,y \in V^-\,,
  \label{j-non-r-flux}
\end{align}
and equals zero for other edges in $E$.
From (\ref{theta-bar}) and (\ref{p-bar-path}), $\bar{J}$ can be written as a path ensemble average
\begin{align}
\bar{J}(x\rightarrow y) = \langle \sum\limits_{l=0}^\sigma \mathbf{1}_x(x_l)
\mathbf{1}_y(x_{l+1})\rangle\,,
\label{j-bar-path}
\end{align}
i.e. the average number of times that edge $x\rightarrow y$ is visited by the nonreactive
trajectories. Applying Proposition~\ref{prop-2} and Proposition~\ref{prop-3}, we have
\begin{corollary}
  \begin{align}
    \begin{split}
      &\sum\limits_{x \in V^-} \bar{J}(x\rightarrow y)  = \bar{\theta}(y)  - \mu(y)
    \mathbf{1}_A(y)\,, \qquad \forall y \in V^-\,,\\
    &\sum\limits_{y \in V^-} \bar{J}(x\rightarrow y)  = \bar{\theta}(x)\Big[1 -
  \mathbf{1}_A(x)\sum_{z \in V}p(z\,|\,x) q(z)\Big] = \bar{\theta}'(x) \,, \qquad \forall x \in V^-\,.
  \end{split}
    \end{align}
    \label{corollary-0}
\end{corollary}
    \subsection{Reactive ensemble}
    \label{subsec-reactive}
    In this subsection, we turn to study the reactive ensemble $\Xi_r$ in (\ref{ensembles}).
First of all, from the definition (\ref{ensembles}), we know
    that the probability distribution of the initial state $x_\sigma$ of the ensemble $\Xi_r$
    coincides with the probability distribution $\mu_r$ of the end node in the ensemble
    $\Xi_{non}$. See~(\ref{mu-r}) and Proposition~\ref{prop-3} (an alternative derivation can be found in Appendix~\ref{appsec-1}).

    Now recall the definition of set $V^+$ in (\ref{v-minus-plus}),
    Proposition~\ref{prop-0}, and also notice that $A \subseteq V^+$.
In analogy to the previous subsection, our aim is to construct a Markov jump process on space $V^+\cup B$ whose path ensemble coincides
  with $\Xi_r$. To do so, we define the transition probabilities
  \begin{align}
\widetilde{p}(x\,|\,y) =
\left\{
  \begin{array}{ll}
				      \frac{p(x\,|\,y)q(x)}{q(y)}\,,  &
    \mbox{if}~y \in
    V^+ \cap A^c\,, \\
    \frac{p(x\,|\,y)\,q(x)}{\sum\limits_{z \in V}p(z\,|\,y)\,q(z)}\,, &
    \mbox{if}~y \in V^+\cap A\,,\\
    \delta_y(x)\,, & \mbox{if}~y \in B\,,
  \end{array}
  \right.
  \label{p-tilde}
\end{align}
where $x, y \in V^+ \cup B$, and $\delta_y$ is the delta function centered at
$y$. We have the following result (the proof is omitted since it is similar to
Proposition~\ref{prop-1}; Note that a similar result under the assumption of
ergodicity has been obtained in~\cite{Cameron2014}).
\begin{prop}
  The ensemble $\Xi_{r}$ can be generated from the trajectories of the Markov jump
  process on space $V^+\cup B$ which is defined by the transition
  probabilities $\widetilde{p}$ in (\ref{p-tilde}) and the initial
  distribution $\mu_r$ on $A$.
  \label{prop-4}
\end{prop}

For $x \in V$, let $\widetilde{\theta}(x)$ be the average number of times that node $x$ has been visited
by reactive trajectories, i.e.
\begin{align}
  \widetilde{\theta}(x) = \langle \sum_{l=\sigma}^{\tau_B} \mathbf{1}_x(x_l)\rangle\,,
  \label{theta-tilde}
\end{align}
where $(x_0, x_1, \cdots, x_\sigma, \cdots, x_{\tau_B})$ is a first passage
path and $\langle\cdot \rangle$ denotes the corresponding ensemble
average. Clearly $\widetilde{\theta}(x) = 0$ when $x \not\in V^+ \cup B$. Summing up $x
\in V$ in (\ref{theta-tilde}), we can obtain
  \begin{align}
    \sum_{x \in V} \widetilde{\theta}(x) = \langle \tau_B \rangle - \langle \sigma \rangle + 1\,.
    \label{theta-tilde-sum}
  \end{align}
  Similar to Proposition~\ref{prop-2}, we have
\begin{prop}
  Let $\widetilde{p}$ be the transition probability defined in (\ref{p-tilde}). We have $\widetilde{\theta}(x) = \mu_r(x)$ for $x \in A$,
  and
\begin{align}
  \widetilde{\theta}(x) = \sum_{y \in V^+} \widetilde{\theta}(y)
  \widetilde{p}(x\,|\,y)\,, \quad  \forall x \in A^c\,.
  \label{theta-tilde-2}
\end{align}
Especially, $\sum\limits_{x \in B} \widetilde{\theta}(x) = 1$.
\label{prop-5}
\end{prop}
\begin{proof}
  Since nodes of set $A$ appear in reactive trajectories only as the starting
  state $x_\sigma$ with probability distribution $\mu_r$, we have $\widetilde{\theta} = \mu_r$ on $A$.
  Equation (\ref{theta-tilde-2}) can be proved in an analogous way to Proposition~\ref{prop-2}, by rewriting $\widetilde{p}$ as the ensemble
  average and applying Proposition~\ref{prop-4}. See (\ref{p-bar-path}).
  By the definition of the stopping time $\tau_B$, we know $x_l
  \not\in B$ for $l < \tau_B$. Therefore, from (\ref{theta-tilde}),
  \begin{align*}
    \widetilde{\theta}(x) = \langle \mathbf{1}_x(x_{\tau_B})\rangle , \quad x \in B\,.
  \end{align*}
  Summing up $x \in B$ and using the fact that $x_{\tau_B} \in B$, we conclude
  \begin{align*}
    \sum_{x \in B} \widetilde{\theta}(x) = \langle\sum_{x \in B}\mathbf{1}_x(x_{\tau_B})\rangle = 1\,.
  \end{align*}
\end{proof}

\begin{remark}
  Notice that (\ref{theta-tilde-2}) holds for node $x \in B$ as well. For numerical computation,
we can first compute $\widetilde{\theta}(x)$ for $x \in V^+\subseteq B^c$ by solving
the linear system (\ref{theta-tilde-2}), and then obtain the end node distribution
$\widetilde{\theta}(x)$ for each $x \in B$ from (\ref{theta-tilde-2}).
\end{remark}

For each edge $x\rightarrow y$, $x \in V^+$, $y \in V^+\cup B$, we define the reactive
probability flux as
\begin{align}
  \widetilde{J}(x\rightarrow y) = \widetilde{\theta}(x) \widetilde{p}(y\,|\,x) = \frac{\widetilde{\theta}(x) p(y\,|\,x) q(y)}{\sum\limits_{z \in
  V} p(z\,|\,x) q(z)}\,,
  \label{J-r-flux}
\end{align}
and set it to zero for other edges in $E$.
In analogy to (\ref{j-bar-path}), we have the ensemble average representation
\begin{align}
\widetilde{J}(x\rightarrow y) = \langle \sum_{l=\sigma}^{\tau_B-1}
\mathbf{1}_x(x_l)\mathbf{1}_y(x_{l+1}) \rangle\,,
\label{j-tilde-path}
\end{align}
where $(x_0, x_1, \cdots, x_\sigma, \cdots, x_{\tau_B})$ is a first passage
path and $\langle \cdot \rangle$ denotes the corresponding ensemble average.
Applying Proposition~\ref{prop-5}, we can obtain
\begin{corollary}
  Let $\widetilde{J}$ be the reactive probability flux defined in
  (\ref{J-r-flux}) and $\widetilde{\theta}$ be given in (\ref{theta-tilde}). We have
  \begin{align}
    \begin{split}
    &\sum\limits_{y \in V^+ \cup B} \widetilde{J}(x\rightarrow y) = \widetilde{\theta}(x)\,, \qquad
    \forall x \in V^+\,,\\
    &\sum\limits_{x \in V^+} \widetilde{J}(x\rightarrow y) = \widetilde{\theta}(y)\,, \qquad
    \forall y \in V^+ \cup B\,, \quad y \in A^c\,, \\
    &\sum_{x \in A} \sum_{y \in V^+ \cup B} \widetilde{J}(x\rightarrow y) = \sum_{x \in V^+} \sum_{y \in B} \widetilde{J}(x\rightarrow y) = 1\,.
  \end{split}
\end{align}
\end{corollary}
\begin{proof}
  It follows directly by applying (\ref{J-r-flux}), Proposition~\ref{prop-5},
  and the fact that row sums of matrix $\widetilde{p}$ are one.
\end{proof}

\subsection{First passage path ensemble}
\label{subsec-fpp}
Based on the previous analysis of the nonreactive and reactive trajectories,
we study the whole first passage path ensemble in this subsection.

Given $x \in V$, we define $f(x)$ as the mean first hitting time of set $B$
for the discrete-time Markov jump process
starting from $x$, i.e. $f(x) = \langle \tau_{B,x}\rangle$. Using
Markovianity, we can write it as
\begin{align}
  f(x) = \frac{\langle \sum\limits_{l=0}^{\tau_B} \mathbf{1}_x(x_l)(\tau_B -
  l)\rangle}{\langle \sum\limits_{l=0}^{\tau_B} \mathbf{1}_x(x_l)\rangle}\,.
  \label{f-path-exp}
\end{align}
It is also well known that $f$ satisfies the equations
\begin{align}
  \begin{split}
    &\sum_{y \in V} p(y\,|\,x) f(y) - f(x) = -1\,, \quad x \in B^c\,, \\
    &f(x) = 0\,,\quad  x \in B\,.
\end{split}
\label{mfpt-eqn}
\end{align}
Accordingly, if the probability distribution of the initial state on set $A$ 
is $\mu$, the average length of the first passage path is then given by $\sum\limits_{x \in A} \mu(x) f(x)$.
Also, let $\theta(x)$ be the average number of times that node $x$ has been visited by the
first passage path with initial distribution $\mu$, i.e.
\begin{align}
  \theta(x) = \langle \sum_{l=0}^{\tau_B} \mathbf{1}_x(x_l)\,\rangle .
  \label{theta}
\end{align}
In analogy to Proposition~\ref{prop-2} and Proposition~\ref{prop-5}, we know it satisfies
\begin{align}
  \begin{split}
    &\theta(x) = \mu(x) \mathbf{1}_A(x) + \sum_{y \in B^c} \theta(y) p(x\,|\,y)\,,\quad x
  \in V\,, \\
  &\sum_{x \in B} \theta(x) = 1\,.
\end{split}
\label{theta-eqn}
\end{align}

Furthermore, taking into account the definitions of $\bar{\theta}$, $\bar{\theta}'$ and $\widetilde{\theta}$ in
(\ref{theta-bar}), (\ref{theta-bar-prime})
and (\ref{theta-tilde}), as well as relations
(\ref{theta-bar-sum}), (\ref{theta-tilde-sum}), we
can verify the following relations
\begin{align}
  \begin{split}
    &\theta(x) = \bar{\theta}(x) + \widetilde{\theta}(x) - \mu_r(x)
    \mathbf{1}_A(x) = \bar{\theta}'(x) + \widetilde{\theta}(x)\,, \quad \forall x \in V\,,\\
    &\sum_{x \in A} \mu(x) f(x) = \sum_{x \in B^c} \theta(x) = \sum_{x \in B^c} \bar{\theta}(x) + \sum_{x \in B^c}
    \widetilde{\theta}(x) - 1 =
    \sum_{x \in B^c} \bar{\theta}'(x) + \sum_{x \in B^c} \widetilde{\theta}(x) \,. \\
\end{split}
\label{f-theta-bar-theta-tilde}
\end{align}
In fact, the following explicit relations hold.
\begin{prop}
  For any $x \in V$,
  \begin{align}
    \bar{\theta}(x) = \theta(x) (1 - q(x)), \quad
    \widetilde{\theta}(x) = \theta(x) q(x) + \mu_r(x) \mathbf{1}_A(x) \,.
\label{theta-bar-mr-m}
\end{align}
\label{prop-6}
\end{prop}
\begin{proof}
  Clearly $\bar{\theta}(x) = \widetilde{\theta}(x) = 0$ if $\theta(x)=0$. Hence we only need to consider the case
  when $\theta(x) > 0$.
  Using the Markov property, we can compute the probability
  $\mathbb{P}(\tau_{A,x} < \tau_{B, x})$ from the first passage ensemble.
  Specifically, let $(x_0, x_1, \cdots, x_{\tau_B})$ be a first passage path.
Conditioning on $x_l = x$ where $0 \le l \le \tau_B$, the event
$\{\tau_{A,x} < \tau_{B,x}\}$ is actually equivalent to $\{l \le \sigma\}$.   Therefore,
  \begin{align*}
    1 - q(x) =& \mathbb{P}(\tau_{A,x} < \tau_{B,x}) \\
    =& \mathbb{P}\Big( l \le \sigma \,\Big|\, x_l = x, 0 \le l \le \tau_B\Big) \\
    =& \frac{\mathbb{P}\Big( l \le \sigma \,, x_l = x, 0 \le l \le
  \tau_B\Big)}
  {\mathbb{P}\Big(x_l = x, 0 \le l \le \tau_B\Big)}
  = \frac{\big\langle \sum\limits_{l=0}^{\sigma}
\mathbf{1}_x(x_l)\big\rangle}{\big\langle \sum\limits_{l=0}^{\tau_B} \mathbf{1}_x(x_l)\big\rangle}
= \frac{\bar{\theta}(x)}{\theta(x)}\,,
  \end{align*}
  which implies $\bar{\theta}(x) = \theta(x)\big(1-q(x)\big)$.
  The expression of $\widetilde{\theta}(\cdot)$ follows from (\ref{f-theta-bar-theta-tilde}).
\end{proof}

We also introduce the probability flux for edge $x\rightarrow y$, which is defined as
\begin{align}
  J(x\rightarrow y) = \theta(x) p(y\,|\,x)\,.
  \label{j-def}
\end{align}
As in (\ref{j-bar-path}) and (\ref{j-tilde-path}), we have
\begin{align}
J(x\rightarrow y) = \langle \sum_{l=0}^{\tau_B-1} \mathbf{1}_x(x_l)\mathbf{1}_y(x_{l+1}) \rangle\,.
\label{j-flux-path}
\end{align}
And the following result is straightforward.
\begin{prop}
  For $x \in B^c$, we have $\theta(x) p(y\,|\,x) = \bar{\theta}(x) \bar{p}(y\,|\,x) +
  \widetilde{\theta}(x) \widetilde{p}(y\,|\,x)\,$, or equivalently,
  \begin{align}
    J(x\rightarrow y) = \bar{J}(x\rightarrow y) + \widetilde{J}(x \rightarrow y)\,.
    \label{J-flux}
  \end{align}
\end{prop}
\begin{proof}
  (\ref{J-flux}) follows directly from the ensemble average representations of
  fluxes in (\ref{j-flux-path}), (\ref{j-bar-path}), (\ref{j-tilde-path}).
\end{proof}

\section{Analysis of the first passage path ensemble : continuous-time jump processes}
\label{sec-ctjp}
In this section, we turn to study the first passage path ensemble of continuous-time
jump processes with general waiting time distributions. After introducing the
process and some notations, we will derive a discrete-time Markov jump process which encodes the dynamical information of the original continuous-time
process. Applying the analysis in Section~\ref{sec-reactive} to this discrete-time Markov jump process, we are able to analyze the first
passage path ensemble of the continuous-time jump process.
\subsection{Preparations}
\label{subsec-continuous-to-discrete}
We consider a continuous-time jump process on network $G=(V,E)$.
We will follow the derivations in~\cite{non_poisson_jump} and some extra notations are needed in order to introduce the process.
A related study on the continuous-time jump processes can also be found in~\cite{path_mem_carse_graining}.

Following Subsection~\ref{subsec-prepare},
let us first assume that node $x \not\in \mathcal{T}$ and consequently $\mathcal{N}_x
\not=\emptyset$.
From node $x$, the system may jump to one of the nodes in
$\mathcal{N}_x$ after staying at $x$ for a certain duration of time. We assume
that the jump events at each node $x$ are independent of each other, and
denote the probability density of the waiting
time $t$ as $\psi(t\,|\,x\rightarrow y)$, conditioning on that the system jumps from $x$ to another node $y
\in \mathcal{N}_x$. As a probability density function, it satisfies $\int_0^{+\infty} \psi(t\,|\,x\rightarrow y)\, dt = 1$.
    Let $a(t\,|\,x)$ be the probability that the system stays at
node $x \in V$ for a time longer than $t$ before it leaves.
Also denote the probability density that the system jumps from node $x$ to $y$
at time $t$ by $b(t\,,x\rightarrow y)$. Since we have assumed that the jump events to different
  nodes are independent of each other, it is straightforward to verify that
\begin{align}
  \begin{split}
  a(t\,|\,x) =& \prod_{y \in \mathcal{N}_x} \int_t^{+\infty} \psi(u\,|\,x\rightarrow y)\, du\,,\\
  b(t\,,x\rightarrow y) =& \Big(\prod_{z \in \mathcal{N}_x, z \neq y}
    \int_t^{+\infty} \psi(u\,|\,x\rightarrow z) du\Big)
    \psi(t\,|\,x\rightarrow y)\,, \quad \forall\, y \in \mathcal{N}_x\,,
\end{split}
\label{a-b-formula}
\end{align}
and they satisfy the equation
\begin{align}
a(t\,|\,x) + \sum\limits_{y \in \mathcal{N}_x} \int_0^t b(u\,,x \rightarrow
y)\, du = 1\,,
\qquad \forall\,  t \ge 0\,. \label{a-b-identity}
\end{align}
We also know that the average time the system will stay at a node $x \in V$ is
\begin{align}
\kappa(x) = \int_0^{+\infty} a(u\,|\,x)\, du\,.
\label{kappa}
\end{align}

Let $p(y\,|\,x)$ be the probability that the system jumps from node $x$ to $y$,
    regardless of when the jump event occurs, then we have
    \begin{align}
      p(y\,|\,x) = \int_0^{+\infty} b(u\,,x\rightarrow y)\, du\,, \quad y \in
      \mathcal{N}_x\,,
    \label{p-b-formula}
  \end{align}
  and $p(y\,|\,x) = 0$, for $y \not\in \mathcal{N}_x$.
From (\ref{a-b-formula}) it follows $\lim\limits_{t\rightarrow +\infty} a(t\,|\,x) = 0$
  and, taking the limit $t \rightarrow +\infty$ in (\ref{a-b-identity}), we can obtain
\begin{align}
  \sum\limits_{y \in V} p(y\,|\,x) = \sum\limits_{y \in \mathcal{N}_x}
  p(y\,|\,x) = 1\,, \qquad \forall x \not\in \mathcal{T}\,.
\label{p-sum-1}
\end{align}

We also assume that the process will terminate once it reaches one of the sink nodes.
Correspondingly, for $x \in \mathcal{T}$, we set $p(y\,|\,x) = \delta_{x}(y)$, $y \in V$,
and therefore the equality $\sum_{y \in V} p(y\,|\,x) = 1$ is still satisfied.
In other words, $p$ is a probability matrix and therefore defines a discrete-time Markov
jump process on $G$. As further presented in Appendix~\ref{appsec-2}, this
discrete-time Markov jump process encodes the dynamical information
of the original continuous-time process.  And it will be useful in the following Subsection~\ref{subsec-continuous-analysis}.

    For general probability densities $\psi$, numerical integrations are
    needed in
    order to compute transition probabilities $p(y\,|\,x)$ from (\ref{a-b-formula}) and (\ref{p-b-formula}).
    In the following, we discuss three special cases~\cite{burst_mem,steady_state_mean_recurrence2015}
    when the analytical expressions of
    $p(y\,|\,x)$ can be obtained (see Figure~\ref{fig-dist-cmp}).
      \begin{enumerate}
	\item
	  Exponential distributions.
	The probability densities are
\begin{align}
\psi(t\,|\,x\rightarrow y) = \lambda_{xy} e^{-\lambda_{xy}t}\,,
\label{exp-dist}
\end{align}
for some $\lambda_{xy} > 0$, where $y \in \mathcal{N}_x$. It corresponds to
the continuous-time Markov jump process
and we can define the generator matrix $L$ of the process
whose entries are
\begin{align}
  L_{xy} = \left\{
    \begin{array}{cl}
      \lambda_{xy}>0\,, & y \in \mathcal{N}_x\,, \\
	     0\,, & y \not\in \mathcal{N}_x, \,y \neq x \,, \\
-\sum\limits_{z\neq x}\lambda_{xz}\,, & y=x\,,
    \end{array}
    \right.
\end{align}
for $x \not\in \mathcal{T}$,
and $L_{xy} = 0$, $\forall y \in V$, when $x \in \mathcal{T}$. From
(\ref{a-b-formula}), (\ref{p-b-formula}) we can obtain that, for $x \not\in
\mathcal{T}$,
    \begin{align}
      \begin{split}
	&a(t\,|\,x) = e^{-(\sum_{z\in \mathcal{N}_x} \lambda_{xz}) t}\,,  \quad
 b(t\,,x\rightarrow y) = \lambda_{xy}e^{-(\sum_{z\in \mathcal{N}_x} \lambda_{xz}) t} \,, \quad t \ge
 0\,, \\
 & \kappa(x) = \frac{1}{\sum_{z\in \mathcal{N}_x} \lambda_{xz}}\,, \quad
 p(y\,|\,x) = \frac{\lambda_{xy}}{\sum_{z\in \mathcal{N}_x}
 \lambda_{xz}}=-\frac{\lambda_{xy}}{L_{xx}}\,, \quad y \in \mathcal{N}_x\,.
      \end{split}
      \label{exp-formulas}
  \end{align}
	\item
	  Weibull distributions.
	  In this case, we assume
	  \begin{align}
	    \psi(t\,|\,x\rightarrow y) = k\lambda_{xy}
	  \big(\lambda_{xy} t\big)^{k-1} e^{-(\lambda_{xy}t)^k}\,,
	  \label{weibull}
	\end{align}
	  for some $k > 0$ and $\lambda_{xy} > 0$, where $y \in \mathcal{N}_x$.
From (\ref{a-b-formula}), (\ref{p-b-formula}) we can obtain that, for $x \not\in
\mathcal{T}$,
    \begin{align}
      \begin{split}
	&a(t\,|\,x) = \exp\Big(-t^k\sum_{z\in \mathcal{N}_x}
	\lambda_{xz}^k\Big)\,,\\
	& b(t\,,x\rightarrow y) = k\lambda_{xy}\big(\lambda_{xy}
      t\big)^{k-1}\exp\Big(-t^k\sum_{z\in \mathcal{N}_x} \lambda_{xz}^k\Big) \,, \\
 & \kappa(x) = \frac{1}{k \big(\sum_{z \in \mathcal{N}_x}
 \lambda_{xz}^k\big)^{\frac{1}{k}}} \Gamma\Big(\frac{1}{k}\Big)\,, \quad
 p(y\,|\,x) = \frac{\lambda^k_{xy}}{\sum_{z\in \mathcal{N}_x}
 \lambda_{xz}^k}\,, \quad y \in \mathcal{N}_x\,,
      \end{split}
      \label{weibull-formulas}
    \end{align}
    where $\Gamma(z)=\int_0^{+\infty} x^{z-1}e^{-x} dx$ is the Gamma function.  Clearly, we recover the exponential distribution when $k=1$.
	\item
	  Power law distributions.
	  In this case, we assume
	  \begin{align}
	    \psi(t\,|\,x\rightarrow y) = \alpha_{xy}  ( 1+
	    t)^{-(\alpha_{xy}+1)}\,,
	  \label{power-law}
	\end{align}
	  for some $\alpha_{xy} > 1$, where $y \in \mathcal{N}_x$.
From (\ref{a-b-formula}), (\ref{p-b-formula}) we can obtain that, for $x \not\in
\mathcal{T}$,
    \begin{align}
      \begin{split}
	&a(t\,|\,x) = (1+t)^{-\sum_{z\in \mathcal{N}_x}\alpha_{xz}}\,,\\
	& b(t\,,x\rightarrow y) = \alpha_{xy} (1+t)^{-\big(1 + \sum_{z\in
    \mathcal{N}_x} \alpha_{xz}\big) }\,, \quad t \ge 0\,, \\
    & \kappa(x) = \frac{1}{\big(\sum_{z \in \mathcal{N}_x}\alpha_{xz}\big) - 1}\,,
\quad
p(y\,|\,x) = \frac{\alpha_{xy}}{\sum_{z\in \mathcal{N}_x} \alpha_{xz}}\,, \quad y \in \mathcal{N}_x\,.
      \end{split}
      \label{power-formulas}
    \end{align}
    \end{enumerate}
\begin{figure}[htp]
\centering
\includegraphics[width=10cm]{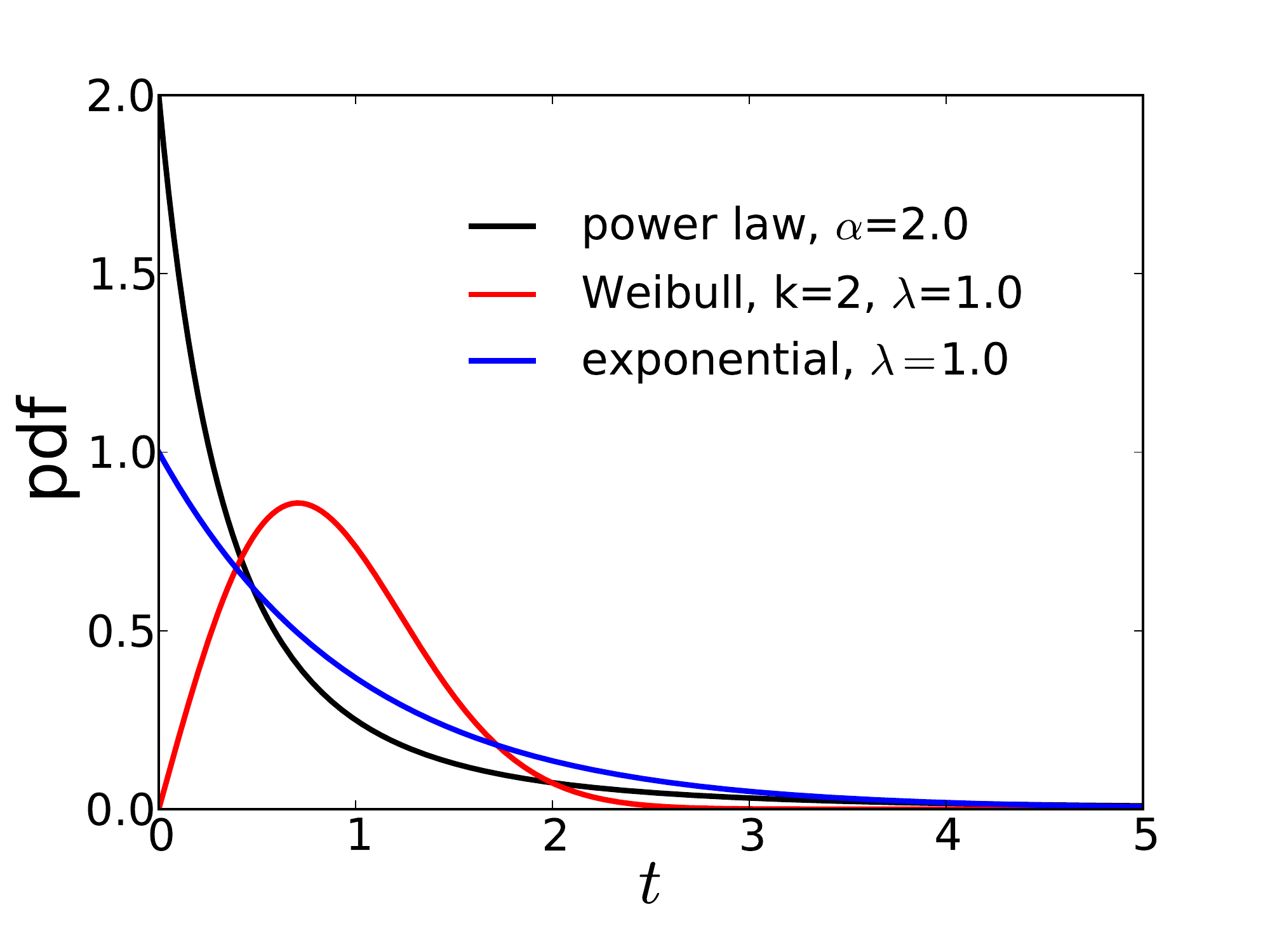}
\caption{Probability density functions $\psi$ of the exponential
distribution, the Weibull distribution as well as the power law
distribution. The parameters are used in (\ref{exp-dist}), (\ref{weibull}) and (\ref{power-law}). \label{fig-dist-cmp}}
\end{figure}
\subsection{Path ensemble analysis}
\label{subsec-continuous-analysis}
In this subsection, we study the path ensemble of the continuous-time
jump process by applying the analysis in Section~\ref{sec-reactive} to the
discrete-time Markov jump process obtained in the previous Subsection~\ref{subsec-continuous-to-discrete}.
The structure of the contents is the same as in Section~\ref{sec-reactive}
and therefore only the necessary steps in order to adapt the analysis to the continuous setting
will be summarized.

\subsubsection*{Nonreactive ensemble}
For the nonreactive ensemble,
denoting $x_s$ as the continuous-time jump process introduced in Subsection~\ref{subsec-continuous-to-discrete}
and following Subsection~\ref{subsec-nonreactive}, we define
\begin{align}
  \bar{T}(x) = \langle \int_0^{t_\sigma} \mathbf{1}_x(x_s) ds\rangle\,,\quad x
  \in V^-\,,
  \label{t-bar}
\end{align}
to be the average amount of time that the nonreactive trajectories spend at
node $x$. Notice that we denote $t_l$ as the time when the $l$th jump occurs and
especially $t_\sigma$ is the last hitting time of set $A$. Recall (\ref{kappa}), we have
\begin{align}
  \bar{T}(x) =& \langle \sum_{l=0}^{\sigma-1} \int_{t_l}^{t_{l+1}}
  \mathbf{1}_x(x_s) ds\rangle \notag \\
  =& \langle \sum_{l=0}^{\sigma-1} (t_{l+1} - t_l)
  \mathbf{1}_x(x_l) \rangle = \kappa(x) \bar{\theta}'(x)\,,
  \label{t-bar-exp}
\end{align}
and therefore the average total amount of time of the nonreactive trajectories is
\begin{align}
  \langle t_{\sigma}\rangle = \sum_{x \in V^-} \bar{T}(x) = \sum_{x \in V^-}
  \kappa(x) \bar{\theta}'(x)\,.
\end{align}
\subsubsection*{Reactive ensemble}
For the reactive ensemble, similar to (\ref{t-bar}), we define
\begin{align}
  \widetilde{T}(x) = \langle \int_{t_\sigma}^{t_{\tau_B}} \mathbf{1}_x(x_s)\, ds\rangle\,,\quad x \in V^+\,,
  \label{t-tilde}
\end{align}
which is the average amount of time that the reactive trajectories spend at
node $x$. From (\ref{kappa}), we have
\begin{align}
  \widetilde{T}(x) =& \langle \sum_{l=\sigma}^{\tau_B - 1} \int_{t_l}^{t_{l+1}}
  \mathbf{1}_x(x_s) ds\rangle \notag \\
  =& \langle \sum_{l=\sigma}^{\tau_B-1} (t_{l+1} - t_l)
  \mathbf{1}_x(x_l) \rangle = \kappa(x) \widetilde{\theta}(x)\,, \quad x \in
  V^+\,.
  \label{t-tilde-exp}
\end{align}
and therefore the average total amount of time of the reactive trajectories is
\begin{align}
  \langle t_{\tau_B} - t_\sigma \rangle = \sum_{x \in V^+} \widetilde{T}(x) = \sum_{x \in V^+}
  \kappa(x) \widetilde{\theta}(x)\,.
\end{align}
\subsubsection*{First passage path ensemble}
For the first passage path ensemble, following Subsection~\ref{subsec-fpp}, we define
\begin{align}
  T(x) = \langle \int_{0}^{t_{\tau_B}} \mathbf{1}_x(x_s) ds\rangle\,,\quad x \in V\,,
  \label{t-total}
\end{align}
which is the average amount of time that the first passage trajectories spend at node $x$.
Similarly as above, we can obtain
\begin{align}
  T(x) =\kappa(x) \theta(x) = \bar{T}(x) + \widetilde{T}(x)\,,
  \label{t-exp}
\end{align}
where $\bar{T}, \widetilde{T}$ are defined in (\ref{t-bar}), (\ref{t-tilde}),
and the average total amount of time of the first passage trajectories is
\begin{align}
  \langle t_{\tau_B} \rangle = \sum_{x \in B^c} T(x) = \sum_{x \in B^c} \kappa(x) \theta(x)\,.
\end{align}

\section{Ergodic case : connections with TPT}
\label{sec-ergodic}
In this section, we consider the special case when either the original process
or, in the continuous-time case, the discrete-time Markov jump process
obtained in Section~\ref{sec-ctjp} is both irreducible and aperiodic. 
In this case, statistics of the transition (reactive) segments have been studied in TPT 
by embedding these segments into an infinitely long stationary trajectory. Our
key observation is that in fact they can be embedded into the first passage
path ensemble with certain initial distribution $\mu$ and therefore can be
studied by applying the analysis in the previous sections. 
This approach elucidates the relations between the study of the first passage paths and the
study of the transition paths in TPT.

To proceed, we first notice that in this case Assumption~\ref{assump-1} in
Section~\ref{sec-reactive} is satisfied and indeed we have
$V^-=V^+= B^c$ in (\ref{v-minus-plus}), $\mathcal{T}=\emptyset$.
The discrete-time jump process is ergodic and we assume its unique invariant measure is
$m$ such that $\sum_{x \in V}m(x) = 1$~\cite{norris-mc,mc-mixing}.
We also introduce the time reversed process defined by
the transition probabilities $p^{-}(y\,|\,x) = \frac{m(y)p(x\,|\,y)}{m(x)}$,
$x,y \in V$ and the
backward committor function $q^{-}$ which satisfies the equation
    \begin{align}
      \begin{split}
	&\sum_{y \in V} p^{-}(y\,|\,x) q^{-}(y) = q^{-}(x)\, , \quad \forall x \in (A\cup B)^c\,, \\
      &q^{-}|_A=1, \quad q^{-}|_B=0\,.
    \end{split}
    \label{committor-backward-ergodic}
    \end{align}
It is known that $q^-(x)$ equals to the probability that, being at $x \in V$, the
process came from set $A$ rather than from set $B$~\cite{tpt_jump,tpt_eric2006}.

We can further assume that there is an infinitely long trajectory $x_0,
x_1, \cdots, x_l, \cdots$ of the discrete-time jump process, where $x_0 \in
V$, $x_0 \sim m$. To make a connection with the general case studied in
the previous subsections, we introduce the stopping times
$\tau^{(0)}_A = \tau^{(0)}_B \equiv 0$, and
\begin{align}
  \begin{split}
  &\tau^{(k+1)}_A = \min\Big\{\,l ~\big|~ l \ge \tau^{(k)}_B, \, x_l \in A\Big\}\,,\\
  &
\tau^{(k+1)}_B = \min\Big\{\,l ~\big|~ l \ge \tau^{(k+1)}_A, \, x_l \in B\Big\}\,,
\end{split}
\label{tau-a-b-k}
\end{align}
where $k\ge 0$. Similarly as in (\ref{sigma-def}) and (\ref{ensembles}), we define
\begin{align}
  \sigma^{(k)} = \max\Big\{\,l~\big|~ x_l \in A,\,\tau^{(k)}_A \le l <
  \tau^{(k)}_B\Big\}\,,\quad k \ge 1\,,
\end{align}
and consider the ensembles of the equilibrium nonreactive and reactive segments
\begin{align}
  \begin{split}
  &\Xi_{non} = \Big\{(x_{j}, x_{j+1}, \cdots, x_{\sigma^{(k)}})~\big|~ j =
  \tau^{(k)}_{A}, ~k \ge 1\Big\}\,,\\
  &\Xi_r = \Big\{(x_{\sigma^{(k)}}, x_{\sigma^{(k)}+1}, \cdots, x_{j})~\big|~ j
  =\tau^{(k)}_{B}, ~k \ge 1 \Big\}\,,
\end{split}
      \label{ensembles-ergodic}
    \end{align}
    constructed from the infinitely long trajectory $x_0, x_1, \cdots$.
Using the definition of $q^{-}$, we can obtain the expressions of the various
quantities related to ensembles (\ref{ensembles-ergodic}) in the ergodic setting.
\begin{prop}
  Consider the path ensembles $\Xi_{non}$ and $\Xi_{r}$ in (\ref{ensembles-ergodic}). 
  \begin{enumerate}
    \item
      For the normalization constant, we have 
\begin{align}
  Z = \sum_{x \in A}\sum_{y \in V} m(y) (1 - q^-(y))p (x\,|\,y) = \sum_{x \in
  B}\sum_{y \in V} m(y) q^{-}(y)p(x\,|\,y)\,.
  \label{z-ergodic}
\end{align}
\item
  The initial distribution $\mu$ (on set $A$) of the ensemble $\Xi_{non}$ satisfies
\begin{align}
  \mu(x) = \frac{1}{Z} \sum_{y \in V} m(y) \big(1 - q^{-}(y)\big) p(x\,|\,y)\,, \quad
  x \in A\,.
  \label{mu-ergodic}
\end{align}
\item
Let $\theta$ be defined in (\ref{theta}), we have
  \begin{align}
    \theta(x) = \left\{
      \begin{array}{ll}
	\frac{1}{Z} m(x) q^{-}(x)\,, & x \in B^c\\
				     &\\
	\frac{1}{Z} \sum\limits_{y \in V}m(y)q^{-}(y) p(x\,|\,y)\,, & x \in B\,.
      \end{array}
    \right.
    \label{theta-ergodic}
  \end{align}
\item
  Let $\mu_r$ be the initial distribution of $\Xi_{r}$, and $\bar{\theta}$, $\widetilde{\theta}$ be defined in
  (\ref{theta-bar}), (\ref{theta-tilde}) respectively. We have
  \begin{align}
    \begin{split}
	   \mu_r(x) &= \frac{1}{Z} m(x) \Big[\sum_{z \in V} p(z\,|\,x)
    q(z)\Big]\,,\quad \forall x \in A\,,\\
    \bar{\theta}(x) &= \frac{1}{Z} m(x) q^{-}(x) (1 - q(x)) \,,\quad \forall x \in V\,, \\
    \widetilde{\theta}(x) &=
   \left\{
     \begin{array}{ll}
    \frac{1}{Z} m(x)q^{-}(x) q(x) +
    \frac{1}{Z}m(x)\Big[\sum_{z\in V} p(z\,|\,x)q(z)\Big] \mathbf{1}_A(x)\,,
    & x \in B^c \,,\\
    & \\
	\frac{1}{Z} \sum\limits_{y \in V}m(y)q^{-}(y) p(x\,|\,y)\,, & x \in B\,.
      \end{array}
    \right.
  \end{split}
  \label{other-quant-ergodic}
  \end{align}
  \end{enumerate}
  \label{prop-ergodic}
\end{prop}
\begin{proof}
  \begin{enumerate}
    \item
      We have
      \begin{align}
	\sum_{y \in V} m(y) q^{-}(y) =& \sum_{x \in V}\sum_{y \in V}
	m(y)q^{-}(y) p(x\,|\,y) \notag\\
	=&
\sum_{x \in A}\sum_{y \in V} m(y)q^{-}(y) p(x\,|\,y)
+
\sum_{x \in B}\sum_{y \in V} m(y)q^{-}(y) p(x\,|\,y) \notag \\
&
+
\sum_{x \in (A \cup B)^c}\sum_{y \in V} m(y)q^{-}(y) p(x\,|\,y) \,.
\label{prop-ergodic-eqn-1}
\end{align}
And it follows from (\ref{committor-backward-ergodic}) that
\begin{align*}
  &\sum_{x \in (A \cup B)^c}\sum_{y \in V} m(y)q^{-}(y) p(x\,|\,y) \\
 =& \sum_{x \in (A \cup B)^c}\sum_{y \in V} m(x)q^{-}(y) p^{-}(y\,|\,x)
= \sum_{x \in (A \cup B)^c} m(x)q^{-}(x)\,.
\end{align*}
Therefore, using $q^{-}|_A = 1$, $q^{-}|_B=0$, from (\ref{prop-ergodic-eqn-1}) we obtain
\begin{align*}
  &\sum_{x \in A}\sum_{y \in V} m(y)q^{-}(y) p(x\,|\,y)
+
\sum_{x \in B}\sum_{y \in V} m(y)q^{-}(y) p(x\,|\,y) \\
=&
\sum_{y \in A} m(y) q^{-}(y) + \sum_{y \in B} m(y) q^{-}(y) = \sum_{y \in A}
m(y)\,,
\end{align*}
which implies (\ref{z-ergodic}).
\item
  Using Markovianity of the discrete-time jump process and the definition of
  $q^{-}$, we can derive
  \begin{align*}
    \mu(x) =&\mathbb{P}(x_j = x\,|\, x_{j-1} = y, ~\mbox{coming from}\,B, y \in V) \\
     =& \frac{\mathbb{P}(x_j = x\,, x_{j-1} = y, ~\mbox{coming from}\,B, y \in
  V)}{ \mathbb{P}(x_{j-1} = y, ~\mbox{coming from}\,B, y \in V)}\\
  \propto&\sum_{y \in V} \mathbb{P}(x_j = x\,, x_{j-1} = y\,, ~\mbox{coming from}\,B)
  \\
  \propto & \sum_{y \in V} \mathbb{P}(x_j = x\,|\, x_{j-1} = y)\, \mathbb{P}(x_{j-1} =
  y, ~\mbox{coming from}\,B) \\
  \propto &  \sum_{y \in V}  m(y) (1-q^{-}(y))p(x \,|\,y)\,,
  \end{align*}
  where $\propto$ means ``equal up to a constant'' and the normalization constant is given in (\ref{z-ergodic}).
\item
  Using the equation of $q^-$ in (\ref{committor-backward-ergodic}) and the formula
  of $\mu$ in (\ref{mu-ergodic}),
  it is easy to verify that $\theta$ defined in (\ref{theta-ergodic})
  satisfies equations (\ref{theta-eqn}).
\item
  The expressions of $\bar{\theta}$, $\mu_r$ and $\widetilde{\theta}$ follows
  from (\ref{theta-ergodic}), Proposition~\ref{prop-3} and Proposition~\ref{prop-6}.
  \end{enumerate}
\end{proof}
\begin{remark}
  Using the expression of $\mu_r$ in (\ref{other-quant-ergodic}) and an argument similar
  to the proof of (\ref{z-ergodic}), we could obtain
  \begin{align}
    \begin{split}
    Z =& \sum_{x \in A}\sum_{y \in V} m(y) (1 - q^-(y))p (x\,|\,y) = \sum_{x \in
  B}\sum_{y \in V} m(y) q^{-}(y)p(x\,|\,y)\\
  = &\sum_{x \in A}\sum_{y \in V} m(x) q(y)p (y\,|\,x) = \sum_{x \in
  B}\sum_{y \in V} m(x) (1 - q(y)) p(y\,|\,x)\,.
\end{split}
  \label{z-ergodic-strong}
  \end{align}
  \label{rmk-z}
\end{remark}

Finally, we consider the continuous-time jump process introduced in Section~\ref{sec-ctjp}
and assume its associated discrete-time Markov jump process obtained in
Subsection~\ref{subsec-continuous-to-discrete} is ergodic.
In this case, we consider an infinitely long trajectory of the process as in
(\ref{path-with-time}) in Appendix~\ref{appsec-2} and let $M=M(N)>0$ be the integer such that $\tau_B^{(M)} \le N < \tau_B^{(M+1)}$,
where the stopping times $\tau_B^{(M)}$ are defined in (\ref{tau-a-b-k}).
Following \cite{tpt_jump}, we define the reaction rate $k_{AB}$ between sets
$A$ and $B$ by
\begin{align}
  k_{AB} = \lim_{N\rightarrow +\infty} \frac{M}{t_N}\,.
  \label{rate-k-ab}
\end{align}
We have the following result.
\begin{prop}
  Consider the continuous-time jump process introduced in
  Subsection~\ref{subsec-continuous-to-discrete}. Let $\kappa$ be defined in
  (\ref{kappa}). $p, q$ are the transition
  probabilities and the committor function of the associated discrete-time
  Markov jump process defined in (\ref{p-b-formula}), (\ref{committor-q}),
  respectively.  Suppose that this discrete-time jump process is ergodic with a unique invariant measure $m$. We have
  \begin{align}
    k_{AB} = \frac{\sum_{x \in
    A}\sum_{y \in V} m(x) p(y\,|\,x) q(y)}{\sum_{x \in V} m(x)\kappa(x)}\,.
    \label{k-ab-formula}
  \end{align}
  \label{prop-rate}
\end{prop}
\begin{proof}
  We consider the path in (\ref{path-with-time}) in Appendix~\ref{appsec-2} which jumps from state
  $x_{i}$ to $x_{i+1}$ at time $t_{i+1}$, $i \ge 0$. We also denote its state
  at time $s\ge 0$ as $x_s$.
  Using the Markov property and ergodicity of the discrete-time jump process, we have
  \begin{align}
    &\sum_{x \in A}\sum_{y \in V} m(x) p(y\,|\,x) q(y) \notag \\
    =&  \sum_{x \in A}
    \sum_{y \in V} \mathbb{P}(x_i=x, x_{i+1}=y, x_{i+1}~\mbox{is on reactive
    segment}\,, \mbox{for some } 0 \le i < +\infty) \notag \\
    =& \lim_{N\rightarrow +\infty} \frac{1}{N}\sum_{k=1}^M \sum_{i=0}^{N-1} \sum_{x \in A}
    \sum_{y \in V} \mathbf{1}_x(x_{i})
    \mathbf{1}_{y}(x_{i+1})\mathbf{1}_{\sigma^{(k)}}(i) \notag \\
    =& \lim_{N \rightarrow +\infty} \frac{M}{N}\,. \label{rate-proof-eqn-1}
  \end{align}
  On the other hand, we have
  \begin{align*}
    \kappa(x) =& \lim_{N\rightarrow +\infty} \frac{\sum\limits_{i=0}^{N-1} (t_{i+1} -
    t_i) \mathbf{1}_x(x_i)}{\sum\limits_{i=0}^{N-1} \mathbf{1}_x(x_i)}\\
    =& \lim_{N\rightarrow +\infty} \frac{\frac{1}{N}\int_0^{t_N}
    \mathbf{1}_x(x_s) ds}{\frac{1}{N} {\sum\limits_{i=0}^{N-1}
    \mathbf{1}_x(x_i)}}
    = \frac{1}{m(x)} \lim_{N\rightarrow +\infty}\frac{1}{N}\int_0^{t_N} \mathbf{1}_x(x_s) ds\,.
  \end{align*}
  Therefore,
  \begin{align}
    \sum_{x \in V} m(x)\kappa(x) =  \lim_{N\rightarrow +\infty}\frac{1}{N}\int_0^{t_N} \sum_{x \in V}\mathbf{1}_x(x_s) ds
    = \lim_{N \rightarrow +\infty} \frac{t_N}{N}\,.
    \label{rate-proof-eqn-2}
  \end{align}
  Combining (\ref{rate-proof-eqn-1}) and (\ref{rate-proof-eqn-2}), we can
  conclude that
  \begin{align*}
    k_{AB} = \lim_{N\rightarrow +\infty} \frac{M}{t_N} = \frac{\sum_{x \in
    A}\sum_{y \in V} m(x) p(y\,|\,x) q(y)}{\sum_{x \in V} m(x)\kappa(x)}\,.
  \end{align*}
\end{proof}

\begin{remark}
  We conclude this section with the following remarks.
  \begin{enumerate}
    \item
      From (\ref{z-ergodic-strong}) in Remark~\ref{rmk-z}, we know the
      numerator in the expression (\ref{k-ab-formula}) equals the constant $Z$ and can be replaced by
      the other expressions in (\ref{z-ergodic-strong}).
    \item
      The case when the continuous-time jump process itself is Markovian and ergodic
      has been studied in \cite{tpt_jump}. The invariant measure $\pi$ in
      \cite{tpt_jump} of the continuous-time process is related to $m$ and $\kappa$ in the
      present work by
      \begin{align}
	\pi(x) = \frac{m(x)\kappa(x)}{\sum_{x \in V} m(x)\kappa(x)}\,.
      \end{align}
      Together with formulas in (\ref{exp-formulas}), we can verify that
      Proposition~\ref{prop-ergodic} and Proposition~\ref{prop-rate} are
      accordant with the results in \cite{tpt_jump}. Therefore, we have
      extended the analysis there to more general continuous-time jump processes introduced in
      Subsection~\ref{subsec-continuous-to-discrete}.
  \end{enumerate}
\end{remark}

\section{Algorithmic issues}
\label{sec-algo}
In this section, we briefly discuss some algorithmic issues related to applying
the analysis of Section~\ref{sec-reactive} and Sections~\ref{sec-ctjp} to applications.
\subsection{Summary of the analysis procedure}
\label{sub-sec-algo-summary}
Given a continuous-time or discrete-time jump process on a finite state space, 
the analysis of its first passage paths can be proceeded as follows.
\begin{enumerate}
  \item
    \textbf{Construction of the discrete-time network.}

    The analysis of the current work can be applied to both continuous-time
    and discrete-time jump processes.
    In the case of a continuous-time jump process,
    we have shown in Section~\ref{sec-ctjp}
    how the transition probabilities $p$ of the corresponding
    discrete-time Markov jump process (without time information) are related to the probability
    density function $\psi$ of the waiting times.
    For certain probability density $\psi$, analytical formulas of $p$,
    $\kappa$ as well as other quantities in Section~\ref{sec-ctjp} can be obtained.
    For more general probability densities, however, numerical integration is needed
    in order to compute $p$, $\kappa$ using formulas
    (\ref{a-b-formula}), (\ref{p-b-formula}) and (\ref{kappa}).

  \item
    \textbf{Calculation of the statistics of path ensembles.}
    After obtaining transition probabilities $p$, we can compute various
    statistical quantities of the nonreactive ensemble, reactive ensemble, as
    well as the entire first passage path ensemble. While the
    equations for each quantity have been derived in
    Section~\ref{sec-reactive} and some quantities can be obtained in several ways,
    we suggest to proceed in the order which is summarized in Algorithm~\ref{algo-1}.
    In the ergodic case where the initial distribution $\mu$ is given in (\ref{mu-ergodic}),
    we could either first obtain $q^{-}$ from the backward committor equation
    (\ref{committor-backward-ergodic}) and compute the other quantities using
    Proposition~\ref{prop-ergodic}, or follow the procedure in~\cite{tpt_jump}.

    Among the steps in Algorithm~\ref{algo-1}, the only possible computational difficulties are related to
    solving the linear systems (\ref{committor}) and (\ref{theta-eqn}). Using
    numerical packages such as PETSc~\cite{petsc-web-page}, however, these linear
    systems can be easily solved for (sparse) networks with several thousand
    nodes and therefore the algorithm could be used in a wide range of applications.
    \begin{algorithm}[h]
      \caption{Analysis of path ensembles for jump process on a network\label{algo-1}}
	  \begin{algorithmic}[1]
	    \State
	Compute the committor function $q$, function $\theta$ by solving
	linear systems (\ref{committor}) and (\ref{theta-eqn}), respectively.
	    \State
	Obtain functions $\bar{\theta}$,
	$\bar{\theta}'$,
	$\widetilde{\theta}$ from the simple relations
	(\ref{mu-r-theta-bar}), (\ref{theta-bar-mr-m}) in
	Proposition~\ref{prop-3} and Proposition~\ref{prop-6}.
	    \State
	Compute the transition probabilities $\bar{p}$ and $\widetilde{p}$
	from equations (\ref{p-bar}), (\ref{p-tilde}), respectively.
	And then compute the fluxes $\bar{J}$, $\widetilde{J}$ and $J$
	from (\ref{j-non-r-flux}), (\ref{J-r-flux}) and (\ref{j-def}).
	    \State
	Compute the mean total path lengths from (\ref{theta-bar-sum}),
	(\ref{theta-tilde-sum}) and (\ref{f-theta-bar-theta-tilde}).
	    \State
	In the case that the process is continuous in time, compute $\bar{T}$,
	$\widetilde{T}$ and $T$ from (\ref{t-bar-exp}), (\ref{t-tilde-exp})
	and (\ref{t-exp}).
      \end{algorithmic}
    \end{algorithm}
  \item
    \textbf{Post-processing.}
    The various statistics obtained above allow us to better understand the mechanism
    of the system's transitions from one set to another.
    However, how to interpret and represent these information is a nontrivial question
    and indeed depends on the problem at hand. In some cases, the key 
    interest is to find certain important pathways which the system is
    likely to take~\cite{tpt2010,tpt_eric2006,tpt_jump}.
    Algorithms for computing dominant pathways in the network setting have
    been studied in \cite{tpt_jump,milestoning_network2013,analyzing_milestoning_network2013}.
    In general, however, the first passage paths of the system may
    be typically very long and diffusive. In this case, presenting the
    statistics by identifying a single representative pathway may be
    incomplete or even misleading (see examples in Section~\ref{sec-example}).
    In the current work, instead of discussing in detail specific methods to
    further exploit the statistics, we simply point out that one can rank the
    importance of nodes and edges based on these statistics. Specifically,
    taking the reactive trajectories as an example, 
    \begin{enumerate}
      \item
	By ranking nodes according to committor function $q$, we can figure out
	how ``close'' each node is to the initial set $A$ and to the terminal
	set $B$. (Here the closeness is measured by the committor probability).
      \item
	By ranking nodes according to function $\widetilde{\theta}$, we know which nodes are more frequently visited by the reactive trajectories.
      \item
	By ranking edges according to flux $\widetilde{J}$, we know which edges are
	more frequently visited by the reactive trajectories.
    \end{enumerate}
    In the case of the continuous-time process, ranking can be also based on
    $\widetilde{T}$ so that we could identify nodes on which the process will spend more time.
    Certainly, similar rankings allow us to identify nodes and
    edges which are important for the nonreactive trajectories, as well as for the entire first passage paths.
\end{enumerate}
\subsection{Data-based approach}
\label{subsec-data-based}
In various real-world applications, we may face the situation that the
system's information is not available, i.e. $p$ or $\phi$ are unknown, and only
the trajectories of the system can be observed.
In this case, one usually takes a data-based approach and
constructs a Markov jump process from the available data.
    Specifically, suppose that we have
    obtained $M$ trajectories of the system.
    For the $i$th trajectory, $1\le i \le M$,
    it starts from $x^{(i)}_0 \in A$ at time $t^{(i)}_0=0$
    and jumps to state $x^{(i)}_l$ at time $t^{(i)}_l$, $1 \le l \le \tau_i$,
    before it reaches set $B$ at time $t^{(i)}_{\tau_i}$ for the first time. That
    is, $x^{(i)}_l \in B^c$ for $0 \le l < \tau_i$ and $x^{(i)}_{\tau_i} \in B$.
    Also define $\sigma_i$ to be the last time when the $i$th trajectory visits
    set $A$. It satisfies $0 \le \sigma_i < \tau_i$,
    $x_{\sigma_i}^{(i)} \in A$ and $x_{l}^{(i)} \in A^c$ for $\sigma_i < l \le
    \tau_i$.
    Then a natural way to estimate the function $\kappa$
    and transition probabilities $p$ is given by
    \begin{align}
      \kappa(x) = \frac{\sum\limits_{i=1}^M\sum\limits_{l=0}^{\tau_i-1}
	(t_{l+1}^{(i)} -
      t_{l}^{(i)})\mathbf{1}_x(x_{l}^{(i)})}{\sum\limits_{i=1}^M\sum\limits_{l=0}^{\tau_i-1}
      \mathbf{1}_x(x_{l}^{(i)})}\,, \qquad
      p(y\,|\,x) = \frac{\sum\limits_{i=1}^M\sum\limits_{l=0}^{\tau_i-1}
      \mathbf{1}_x(x_l^{(i)})\mathbf{1}_y(x_{l+1}^{(i)})}{\sum\limits_{i=1}^M\sum\limits_{l=0}^{\tau_i-1}
      \mathbf{1}_x(x_l^{(i)})}\,,
      \label{kappa-p-estimator}
    \end{align}
    where $x, y \in V$, provided that the denominator is nonzero.
      Similarly, for the initial distribution $\mu$, we can estimate
      \begin{align}
      \mu(x) = \frac{1}{M} \sum_{i=1}^M \mathbf{1}_x(x_0^{(i)})\,,
      \label{init-mu-estimator}
    \end{align}
      where $x \in A$.
    Detailed studies related to the reconstruction of the Markov jump
    processes from data can be found in \cite{estimation_Metzner2007,msm_generation} and references therein.

    After estimating $\kappa$ and $p$ from (\ref{kappa-p-estimator}), we
    can follow Algorithm~\ref{algo-1} and the discussions in
    Subsection~\ref{sub-sec-algo-summary} to perform the analysis.
    In fact, direct calculation shows that function $\theta$ and flux $J$ in
    (\ref{theta}) and (\ref{j-flux-path}) are simply given by
    \begin{align}
      \theta(x) = \frac{1}{M} \sum\limits_{i=1}^M\sum\limits_{l=0}^{\tau_i-1}
      \mathbf{1}_x(x_{l}^{(i)})\,, \quad J(x \rightarrow y) = \frac{1}{M}
      \sum\limits_{i=1}^M\sum\limits_{l=0}^{\tau_i-1}
      \mathbf{1}_x(x_{l}^{(i)}) \mathbf{1}_y(x_{l+1}^{(i)})\,.
      \label{right-estimator}
    \end{align}
    In other words, instead of solving the linear system (\ref{theta-eqn}), the
    ensemble average $\theta$ and $J$ of the constructed Markov jump process
    can be obtained by ``counting'' the trajectory data.
    However, we emphasize that this is not true in general for other quantities of ensemble averages.
    For instance, generally, the committor function $q$ and quantity $\bar{\theta}$,
    which satisfy equations (\ref{committor}) and (\ref{theta-bar-eqn}) respectively,
    are different from
    \begin{align}
      q_{data}(x) =
      \frac{\sum\limits_{i=1}^M\sum\limits_{l=\sigma_i+1}^{\tau_i}
      \mathbf{1}_x(x_{l}^{(i)})}{
      \sum\limits_{i=1}^M\sum\limits_{l=0}^{\tau_i}
    \mathbf{1}_x(x_{l}^{(i)})}\,, \quad \bar{\theta}_{data}(x) =
\frac{1}{M} \sum\limits_{i=1}^M\sum\limits_{l=0}^{\sigma_i}
\mathbf{1}_x(x_{l}^{(i)})\,,
\label{wrong-estimator}
    \end{align}
    which are computed by ``counting'' the trajectory data.
    As a simple counterexample, consider graph $G$ with four nodes $V =
    \{0,1,2,3\}$. We define sets $A=\{0\}, B=\{3\}$ and suppose that only $2$
    trajectories $(0,1,2,3)$, $(0,2,0,2,3)$ are observed (time information is ignored for simplicity).
    Then clearly $\sigma_1=0$, $\sigma_2 = 2$, and
    it follows from (\ref{wrong-estimator}) that $q_{data}(1) = 1$, $\bar{\theta}_{data}(1) = 0$.
    On the other hand, from (\ref{kappa-p-estimator}), we have $p(2\,|\,1) > 0,
    p(0\,|\,2) > 0$, and it follows that $q(1) < 1$ as a consequence of
    Proposition~\ref{prop-0}. Similarly, we can also show that $\bar{\theta}(1) > 0$.

    In summary, after obtaining transition probabilities $p$ from
    (\ref{kappa-p-estimator}), one still needs to apply Algorithm~\ref{algo-1} to
    obtain various statistical quantities. In the first step of
    Algorithm~\ref{algo-1}, although function $\theta$ can be calculated from
    (\ref{right-estimator}), it is necessary to compute the committor function $q$ by solving
    the linear system (\ref{committor}).

\section{Numerical examples}
\label{sec-example}
In this section, we study several examples in order to demonstrate
the analysis in the previous sections.
\subsection{Example $1$ : continuous-time jump processes on a simple graph}
\label{subsec-example-simple-graph}
As the first example, we consider continuous-time jump processes on a graph
which consists of $7$ nodes.
As shown in Figure~\ref{fig-ex1-1} (a), we choose sets $A=\{1,2\}$,
$B=\{7\}$, and study jump processes starting from set $A$ until they
reach set $B$. We assume the initial distribution is uniform on $A$, i.e. $\mu(1)
= \mu(2) = 0.5$. Three different cases are considered, i.e. when the
waiting time at each node satisfies (i) an exponential distribution,
(ii) a power law distribution and (iii) a Weibull distribution, respectively (see
Subsection~\ref{subsec-continuous-to-discrete}).

In both cases of the exponential and Weibull
distribution, we take the same rate constants $\lambda_{xy}$ in (\ref{exp-dist}) and
(\ref{weibull}), which are shown in Figure~\ref{fig-ex1-1}\,(a) for each edge
$x\rightarrow y$, where $k=2$ is used for each edge in the Weibull distribution.
In the case of the power law distribution, we choose $\alpha_{xy}$ in (\ref{power-law}) to
be equal to $\lambda_{xy} + 1$ for each edge $x\rightarrow y$, such that the
mean of the waiting time along each edge is the same as in the
case of the exponential distribution.

In each case, starting from the continuous-time jump processes, we first construct the corresponding discrete-time jump processes
using relations (\ref{a-b-formula}), (\ref{p-b-formula}) in
Section~\ref{sec-ctjp}. The transition probabilities $p$ for each edge as well as
the committor function $q$ for each node are shown in
Figure~\ref{fig-ex1-1} (b)-(d). Clearly,
different assumptions on the waiting time distributions of the continuous-time
jump processes result in discrete-time Markov jump processes with different transition probabilities.
In fact, from the formulas of transitions probabilities $p$ in
(\ref{exp-formulas}), (\ref{weibull-formulas}),  we
can conclude that, comparing to the exponential distribution case, the difference between
jump probabilities $p$ along edge $x\rightarrow y$ and edge $x\rightarrow y'$ 
where $\lambda_{xy}\neq \lambda_{xy'}$ will be larger when the waiting times follow Weibull distribution with $k=2$.
Furthermore, it can be observed that comparing to the exponential distributions, the
probabilities to jump from ``left to right'' (e.g. along edges $3\rightarrow 5$,
$4\rightarrow 6$ and $6\rightarrow 7$) are larger in the case of power law
distributions, and are smaller in the case of Weibull distributions ($k=2$).

The ensembles of the nonreactive segments, the reactive segments, as well as
the whole first passage paths are studied following the analysis in
Section~\ref{sec-reactive} and Section~\ref{sec-ctjp}.
The ensemble averages $\bar{\theta}'$, $\widetilde{\theta}$, $\theta$,
$\bar{T}$, $\widetilde{T}$ and $T$ for each node in different cases are
shown in Table~\ref{tab-ex1-1}.
Also see Figure~\ref{fig-ex1-1-1} where these quantities are displayed in the case of the exponential distribution.
We refer to equations (\ref{theta-bar-prime}), (\ref{theta-tilde}), (\ref{theta}),
(\ref{t-bar}), (\ref{t-tilde}), (\ref{t-total}) as well as Table~\ref{tab-notation} for definitions of
these quantities.
In each case, from the values of the quantity $\bar{\theta}'$ we can observe that the system
will typically visit nodes $3,4$ and return to set $A$ a few times before it eventually
arrives at set $B$.
As shown in the columns with label ``Total'' in Table~\ref{tab-ex1-1}, we also see that while in each case the
reactive trajectories contain a similar number of jumps on average ($4.60$--$5.95$, last column rows $2$, $9$, $16$),
there are fewer jumps within the nonreactive trajectories in the case of power
law distribution ($5.77$), but on the other hand there are more jumps in the nonreactive
trajectories in the case of the Weibull distribution ($22.59$).
The same is also true for the first passage paths as well as when we take time
information into account.
These results are accordant with our previous observations based on Figure~\ref{fig-ex1-1}
(b)-(d) that transitions of the system from set $A$ to $B$ are more difficult when the jump time follows the Weibull distribution.
\begin{figure}
  \begin{tabular}{cc}
    \subfigure[Continuous-time jump processes]{\includegraphics[width=0.43\textwidth]{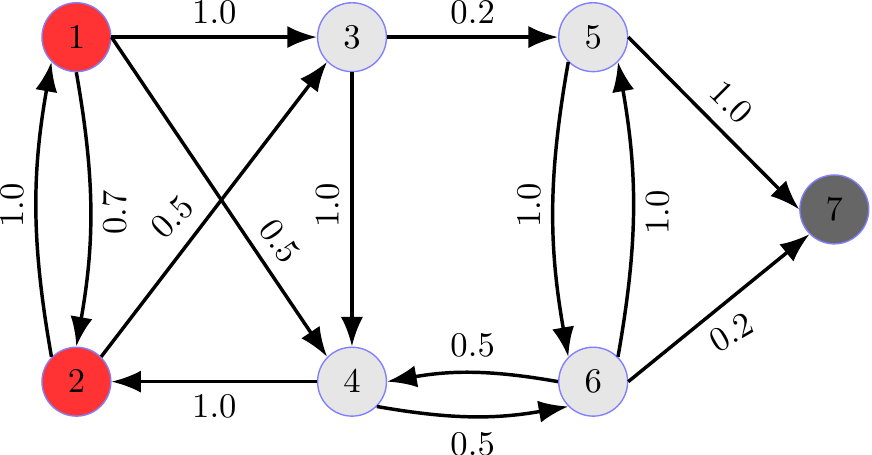}} &
    \subfigure[Exponential distribution]{\includegraphics[width=0.43\textwidth]{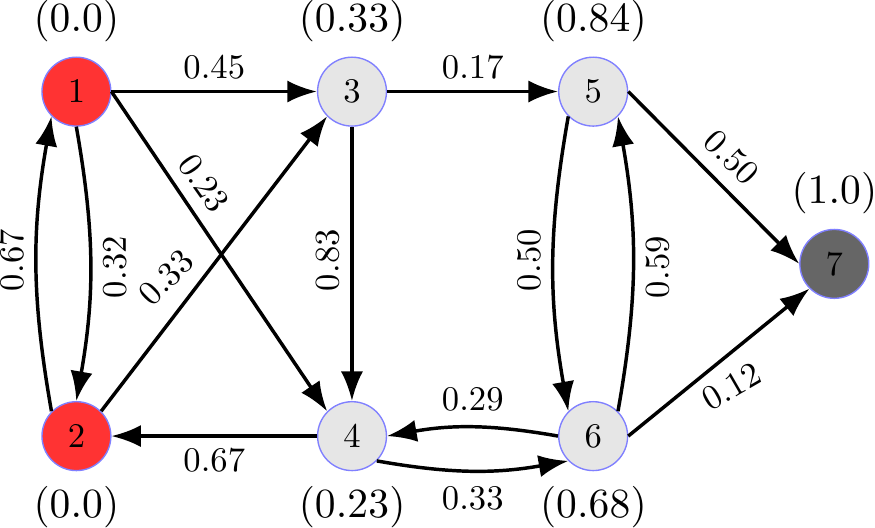}}
    \\
    \subfigure[Power law distribution]{\includegraphics[width=0.43\textwidth]{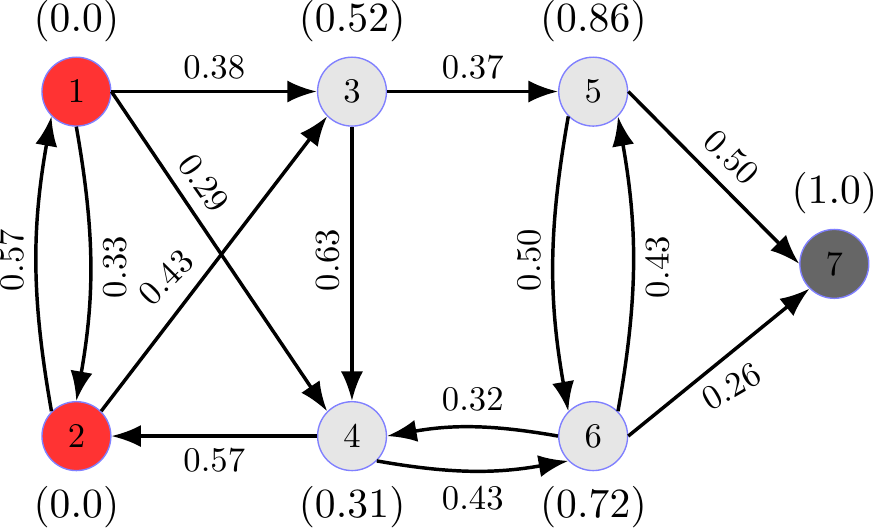}} &
    \subfigure[Weibull distribution]{\includegraphics[width=0.43\textwidth]{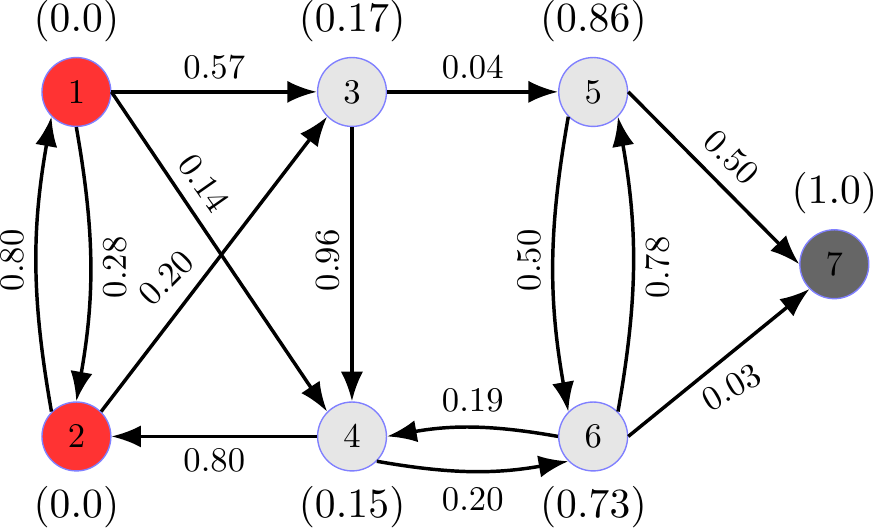}}
\end{tabular}
\caption{Example $1$. Figure (a) : continuous-time jump processes on a graph with $7$ nodes. Sets
  $A=\{1,2\}$ and $B=\{7\}$. Rate constants $\lambda_{xy}$ in both the
exponential and Weibull distributions are shown on each edge.
Figure (b)-(d) : 
the associated discrete-time Markov jump processes
when the jump times of the original
continuous-time process obey exponential distributions, the power law distributions with $\alpha_{xy} = 1
+ \lambda_{xy}$, as well as the Weibull distributions with $k=2$ and
$\lambda_{xy}$, respectively.  
In each case, jump probabilities $p$ of the associated discrete-time Markov jump processes
are shown on each edge, and the values of the committor function $q$ for each node are shown in brackets.  
\label{fig-ex1-1}}
\end{figure}
\begin{figure}
  \begin{tabular}{cc}
    \subfigure[$\bar{\theta}'$ and
    $\bar{J}$]{\includegraphics[width=0.46\textwidth]{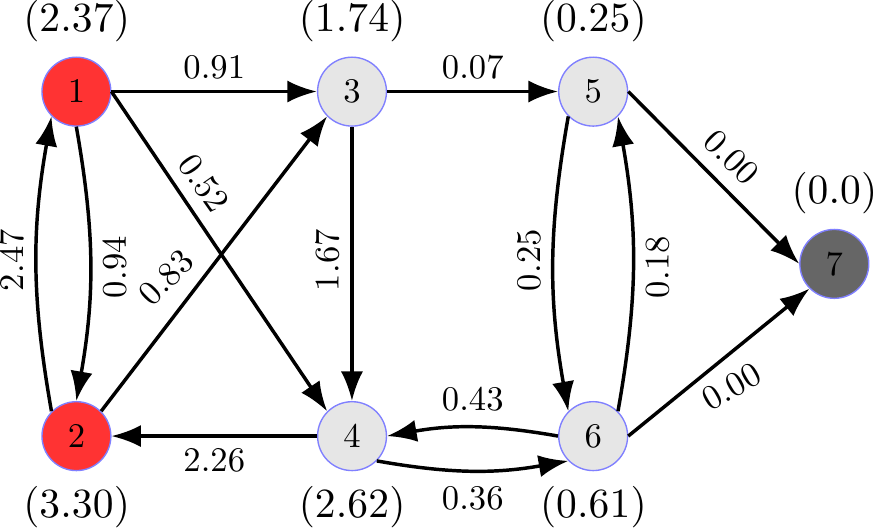}} &
    \subfigure[$\widetilde{\theta}$ and
    $\widetilde{J}$]{\includegraphics[width=0.46\textwidth]{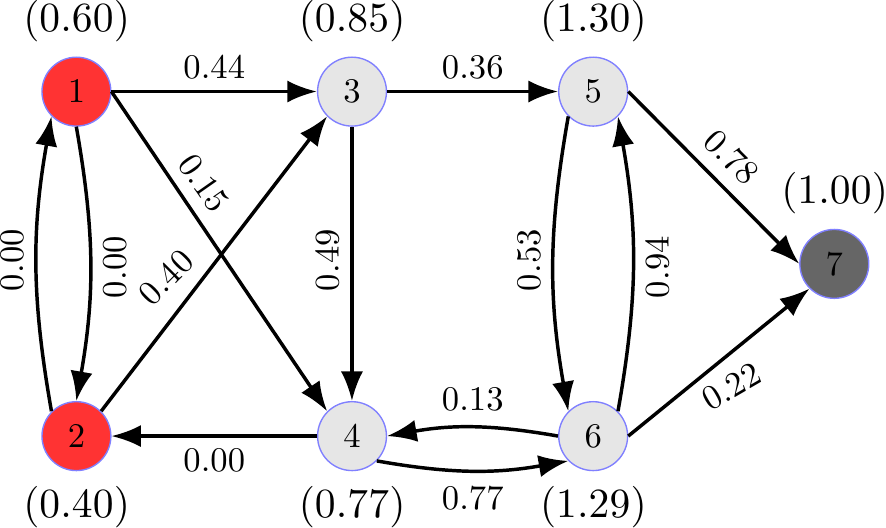}}
\end{tabular}
\caption{Example $1$ in the case of exponential waiting time distributions.
  (a)
   average number of times $\bar{\theta}'$ that each node has been visited by the nonreactive
  trajectories (in the brackets) as well as the nonreactive flux $\bar{J}$ for
  each edge related to the nonreactive trajectories (on the edges).
(b)
average number of times $\widetilde{\theta}$ that each node has been visited
by the reactive trajectories (in the brackets) as well as the reactive flux $\widetilde{J}$ for
each edge related to the reactive trajectories (on the edges).
See Table~\ref{tab-notation} for concrete definitions of these quantities.
\label{fig-ex1-1-1}
}
\end{figure}
\begin{table}[hpt]
\centering
  \begin{tabular}{|c|c|c|c|c|c|c|c|c|c|}
    \hline
     & Node $x$ & $1$ & $2$ & $3$ & $4$ & $5$ & $6$ & $7$ & \mbox{Total} \\
    \hline
   \multirow{7}{*}{exponential} & $\bar{\theta}'$ & $2.37$ & $3.30$ & $1.74$ &  $2.62$ & $0.25$ & $0.61$ & $0.00$ & $10.90$ \\
				       & $\widetilde{\theta}$ & $0.60$ & $0.40$ & $0.85$ &  $0.77$ & $1.30$ & $1.29$ & $1.00$& $5.21$ \\
				 & $\theta$ & $2.97$ & $3.70$ & $2.58$ &  $3.39$ & $1.55$ & $1.91$ & $1.00$& $16.10$ \\
    & $\kappa$ & $0.45$ & $0.67$ & $0.83$ &  $0.67$ & $0.50$ & $0.59$& $-$ & $-$ \\
    &  $\bar{T}$ & $1.08$ & $2.20$ & $1.45$ & $1.75$ & $0.13$ & $0.36$ & $0.00$ & $6.96$ \\
    & $\widetilde{T}$ & $0.27$ &  $0.27$ & $0.71$ & $0.51$ & $0.65$ &  $0.76$& $0.00$ & $3.17$ \\
    & $T$ & $1.35$ & $2.47$ & $2.15$ & $2.26$ & $0.78$ & $1.12$ & $0.00$ & $10.13$\\
    \hline
    \hline
\multirow{7}{*}{power law} & $\bar{\theta}'$ & $1.26$ & $1.73$ & $0.79$ &  $1.39$ & $0.17$ & $0.42$ & $0.00$ & $5.77$ \\
				       & $\widetilde{\theta}$ & $0.51$ & $0.49$ & $0.84$ &  $0.62$ & $1.07$ & $1.07$ & $1.00$ & $4.60$ \\
			 & $\theta$ & $1.77$ & $2.23$ & $1.64$ &  $2.01$ & $1.24$ & $1.48$ & $1.00$ & $10.36$ \\
    & $\kappa$ & $0.24$ & $0.40$ &  $0.45$ & $0.40$ &  $0.33$ & $0.27$ & $-$ & $-$ \\
    & $\bar{T}$ & $0.30$ & $0.69$ & $0.36$ & $0.55$ & $0.06$ & $0.11$ & $0.00$ & $2.08$ \\
    & $\widetilde{T}$ & $0.12$ &  $0.20$& $0.38$ & $0.25$ & $0.36$ & $0.29$ & $0.00$ & $1.59$ \\
    & $T$ & $0.42$ & $0.89$ & $0.74$ & $0.80$ & $0.41$ & $0.40$ & $0.00$ &$3.67$\\
    \hline
    \hline
    \multirow{7}{*}{Weilbull} &  $\bar{\theta}'$ & $5.47$ & $6.91$ & $4.14$ &  $5.23$ & $0.25$ & $0.58$ & $0.00$ & $22.59$ \\
					     &		       $\widetilde{\theta}$ & $0.75$ & $0.25$ & $0.87$ &  $0.89$ & $1.61$ & $1.58$ & $1.00$ & $5.95$ \\
			& $\theta$ & $6.22$ & $7.16$ & $5.01$ &  $6.13$ & $1.87$ & $2.16$ & $1.00$ & $28.54$ \\
    & $\kappa$ & $0.67$ & $0.79$ & $0.87$ & $0.79$ & $0.63$ & $0.78$ & $-$ & $-$ \\
    & $\bar{T}$ & $3.68$ & $5.48$ & $3.60$ & $4.15$ & $0.16$ & $0.46$ & $0.00$ & $17.51$ \\
    & $\widetilde{T}$ & $0.50$  & $0.20$ & $0.76$ & $0.71$ & $1.01$ & $1.23$ & $0.00$ & 	$4.41$ \\
    & $T$ & $4.18$ & $5.67$ & $4.35$ & $4.86$ & $1.17$ & $1.68$ & $0.00$ & $21.92$ \\
    \hline
\end{tabular}
\caption{Example $1$.  Ensemble averages of the continuous-time jump processes and the
  associated discrete-time Markov jump processes are displayed for each node,
  when the jump times obey either exponential distribution, power law distribution,
  or Weibull distribution.
  $\kappa$ is defined in (\ref{kappa}).
$\bar{\theta}'$, $\bar{T}$ are related to the nonreactive trajectories and are
defined in (\ref{theta-bar-prime}),
(\ref{t-bar}).
$\widetilde{\theta}$, $\widetilde{T}$ are related to the reactive trajectories and
are defined in (\ref{theta-tilde}), (\ref{t-tilde}).
$\theta$, $T$ are related to the whole mean first passage paths and are
defined in (\ref{theta}), (\ref{t-total}).
For each ensemble average, the column with label ``Total'' shows the sum of all nodes except node $7$.
Notice that relations (\ref{f-theta-bar-theta-tilde}) and (\ref{t-exp}) hold up to rounding errors.
\label{tab-ex1-1}
}
\end{table}
\subsection{Example $2$ : discrete-time random walk in a maze}
\label{subsec-example-maze}
In this subsection, we consider a (discrete-time) random walk in a maze. This example has
been considered to illustrate TPT in \cite{tpt2010}. However, since the theory
there requires ergodicity of the process, an implicit assumption one need to make is that the walker keeps
moving in the maze even after arriving at the exit cell.
In this work, we investigate the same example (except we don't have to make the aforementioned assumption)
and apply the analysis in
Section~\ref{sec-reactive} to study its path ensembles.

The maze is constructed on a $30 \times 30$ grid and the structure is shown in Figure~\ref{fig-ex2-1}. We assume that the walker
enters the maze from the cell at the bottom left corner. At each unit of time, the walker
moves to a new cell which is chosen randomly with equal probability
from all neighboring cells that are directly reachable from the current one (no wall in between).
The walker keeps moving until it finally arrives at the upper right corner
where it can leave the maze.

To apply the analysis in Section~\ref{sec-reactive}, we construct an
undirected graph whose nodes consist of the cells of the maze
and two nodes are connected by an edge if and only if the corresponding cells
in the maze are adjacent and there is no wall in between.
 Sets $A$, $B$ consist of a single node
corresponding to the cell at the bottom left and the upper right corner (see
Figure~\ref{fig-ex2-1}), respectively.
Since the process itself is a discrete-time Markov jump process,
its transition probabilities $p$ can be directly computed from the
connectivity structure of the maze.
Committor function $q$ and function $\theta$ related to the first
passage paths are shown in Figure~\ref{fig-ex2-2}.
From the committor function $q$, we can observe that there are long vertical
walls which roughly separate the maze into left and right parts. In the left
part (dark blue cells in the left panel of Figure~\ref{fig-ex2-2}), the walker is likely to
return to $A$ first before it goes to $B$. On the other hand, from the plot
of the function $\theta$ in the
right panel of Figure~\ref{fig-ex2-2}, we can realize that, unlike the simple
pathway shown in Figure~\ref{fig-ex2-1}, the first
passage paths from $A$ to $B$ taken by the walker are very diffusive on average, since many
cells have been repeatedly visited within a single first passage path. The
diffusiveness of the paths can also be observed from the very long average path lengths in
Table~\ref{tab-ex234-1}.

In order to further investigate the nonreactive and reactive segments of the
first passage paths, functions $\bar{\theta}'$, $\widetilde{\theta}$ are
computed and are shown in Figure~\ref{fig-ex2-3}. We can see that
both of these two segments are diffusive on average, since many cells are repeatedly visited.
Comparing these two segments, however, it is observed
that a large number of the visits at cells in the left part of the maze
actually belong to the nonreactive segments and do not contribute to the reactive pathways.
Comparing the two panels in Figure~\ref{fig-ex2-3}, we can also conclude that
the nonreactive segments are typically more diffusive than the reactive segments.
Roughly speaking, this phenomena occurs similarly when studying the
transitions of a metastable diffusion process between two local minima of a potential,
where the system spends longer time in the basin of attraction than in the transition region.
\begin{table}
  \centering
  \begin{tabular}{c|c|c|c}
    \hline
    \hline
    & $\bar{L}$ & $\widetilde{L}$ & $L$ \\
    \hline
    Maze & $69704.7$ & $14129.3$ & $83834.0$\\
    Network & $18.01$ & $36.83$ & $54.84$ \\
    Football &  $5.19$ & $3.47$ & $8.65$ \\
    \hline
  \end{tabular}
  \caption{The mean path lengths of the nonreactive segments
    ($\bar{L}$), the reactive segments ($\widetilde{L}$) and the whole first
  passage path ($L$) for the maze, the scale-free network, as well as the football
match examples. \label{tab-ex234-1}}
\end{table}
\begin{figure}
\centering
    \includegraphics[width=0.80\textwidth]{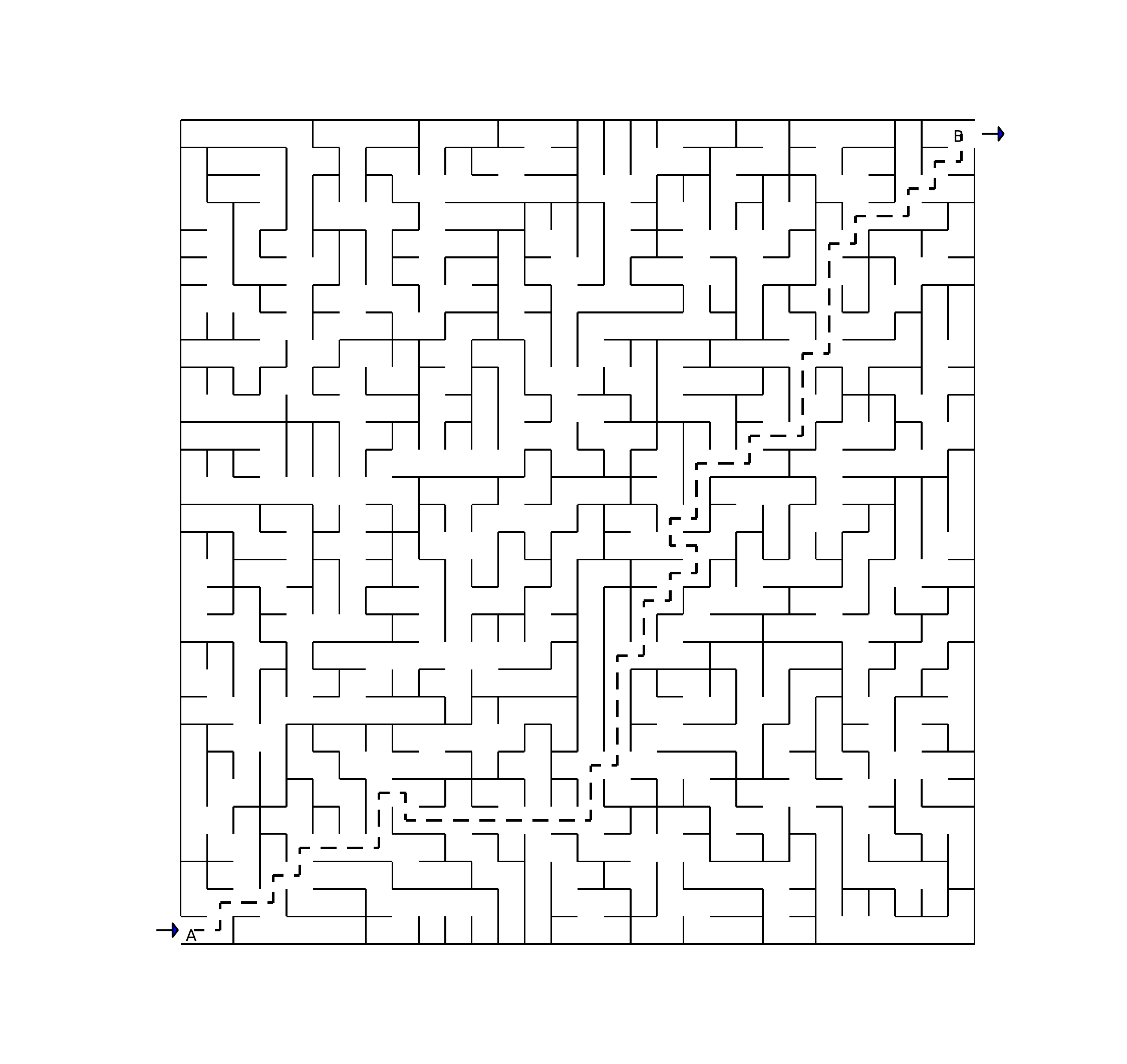}
\caption{Example $2$. Cells
  at the bottom left corner and the upper right corner correspond to set $A$ and $B$,
  respectively. One possible pathway from $A$ to $B$ is shown. \label{fig-ex2-1}}
\end{figure}
\begin{figure}
\centering
    \includegraphics[width=0.85\textwidth]{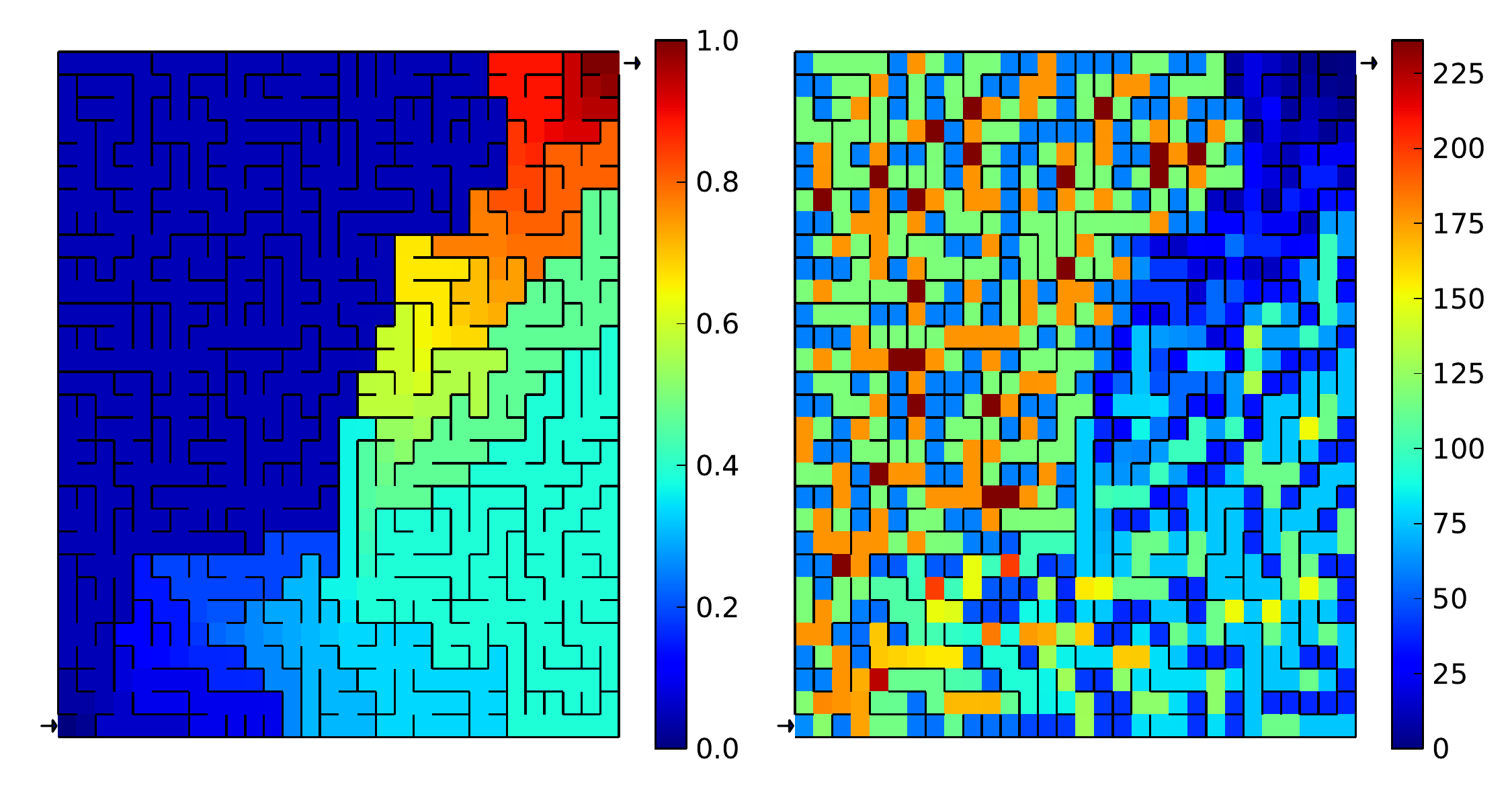}
\caption{Example $2$. Left : values
of committor function $q$ for each cell are shown.
Right : values of function $\theta$ (average number of times that a cell has
been visited by the first passage paths) are shown.
\label{fig-ex2-2}}
\end{figure}
\begin{figure}
\centering
    \includegraphics[width=0.85\textwidth]{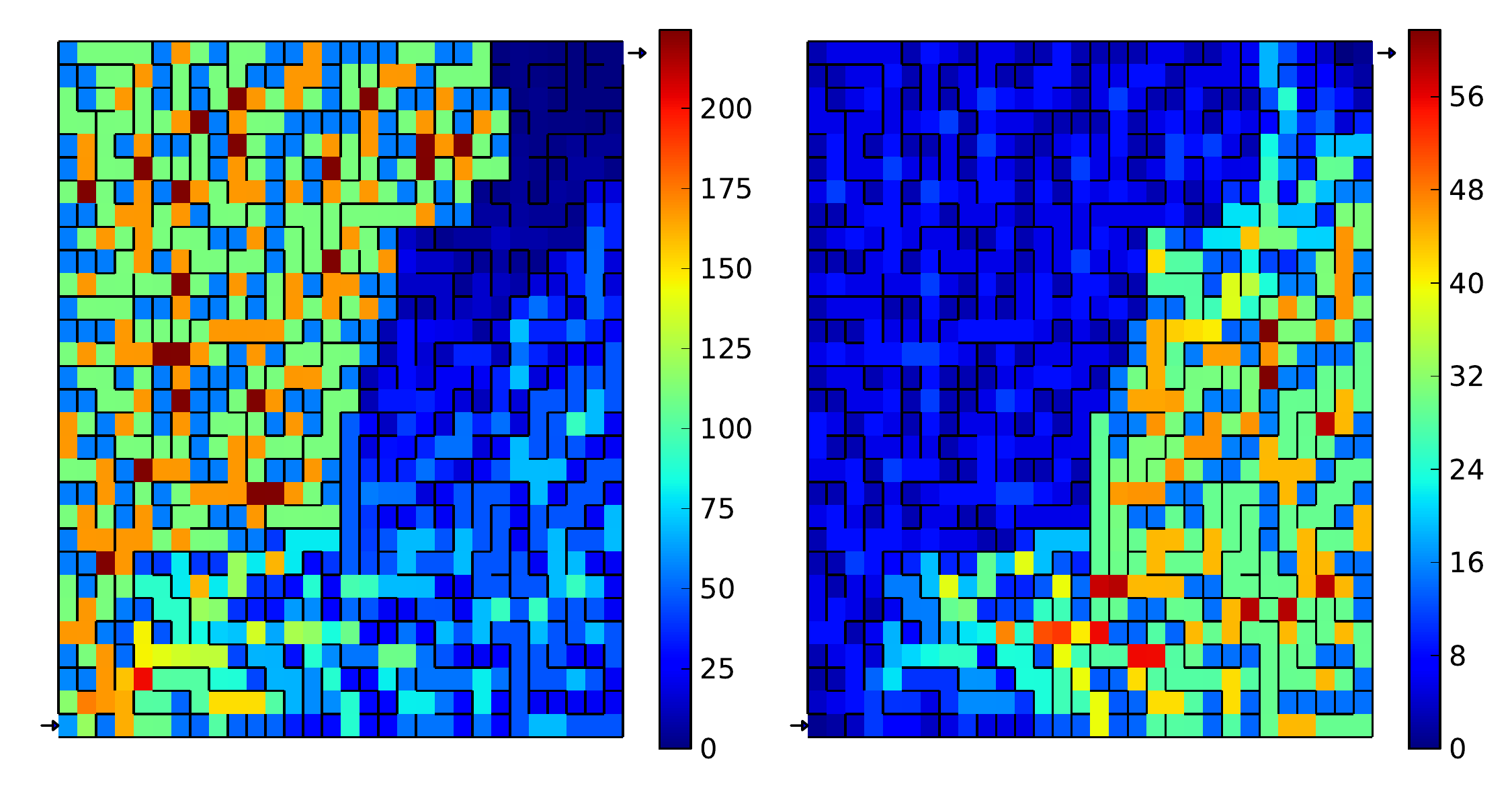}
    \caption{Example $2$. Left : values of function $\bar{\theta}'$ 
      (average number of times that a nonreactive trajectory visits a node)
      corresponding to the nonreactive segments.  
    Right : values of function $\widetilde{\theta}$ 
(average number of times that a reactive trajectory visits a node)
    corresponding to the reactive
  segments. 
\label{fig-ex2-3}}
\end{figure}
\subsection{Example $3$ : discrete-time random walk on a scale-free network}
\label{subsec-network}
In this example, we consider a network with a power law degree distribution.
The network is generated using Python package NetworkX~\cite{ref_networkx}, which
implemented the algorithm of Holme and Kim~\cite{PhysRevE_growing_network} to generate growing graphs with power law
degree distribution. In order to have a better illustration, we choose a
relatively small network with $50$ nodes. During the growth of the network,
two random edges are added for each new node
and the probability of adding a triangle after adding a random edge is $0.1$~\cite{ref_networkx}. The structure of the
network is shown in Figure~\ref{fig-ex3-1}.

To continue, we take the discrete-time random walk on this network as an example and study its first
passage paths starting from node $49$ to node $0$, i.e. $A=\{49\}, B=\{0\}$.
Various statistics are computed following
Algorithm~\ref{algo-1}, as well as the methods presented in Section~\ref{sec-algo}.
In Table~\ref{tab-ex3-1}, nodes are ranked according to different statistics
and part of them are listed. It can be observed that, except for the source node $49$, the committor
values $q$ of all other nodes are larger than $0.5$, reflecting the fact that
these nodes are more tightly connected to node $0$ compared to node $49$,
which is the last node added during the growth of the network.
We can also conclude that nodes with large values $\theta$ tend to
have large values of $\bar{\theta}$ (or $\bar{\theta}'$) and
$\widetilde{\theta}$ as well. Furthermore, a positive
correlation can be observed between the magnitudes
of these statistics and the degrees of nodes, which probably could be explained
by the undirected (symmetrical) movement of the random walker.
In Figure~\ref{fig-ex3-1}, information related to the reactive trajectories
are shown. Specifically, nodes with larger values $q$ are indicated by brighter
colors, and nodes with larger values $\widetilde{\theta}$ are displayed in
larger sizes. Also, edges with larger reactive fluxes $\widetilde{J}$ are plotted
using thicker lines. From Figure~\ref{fig-ex3-1}, we could identify nodes and
edges that are more frequently visited by reactive trajectories.
The average path lengths of the nonreactive trajectories, the reactive
trajectories, as well as the whole first passage paths are shown in
Table~\ref{tab-ex234-1}.
Different from the previous maze example, the lengths of the reactive
trajectories are longer than those of the nonreactive trajectories on average,
which is due to the fact that the attraction of the source set $A$ is not strong.

Summarizing the above results,  we can conclude that, due to the global connectivity of the network,
the path ensembles are widespread in path space and it is difficult
to identify certain dominant transition pathways for this system. But still, we have obtained
insights about the transitions of the random walk both quantitatively and graphically.
\begin{figure}[hptb]
\centering
    \includegraphics[width=14cm]{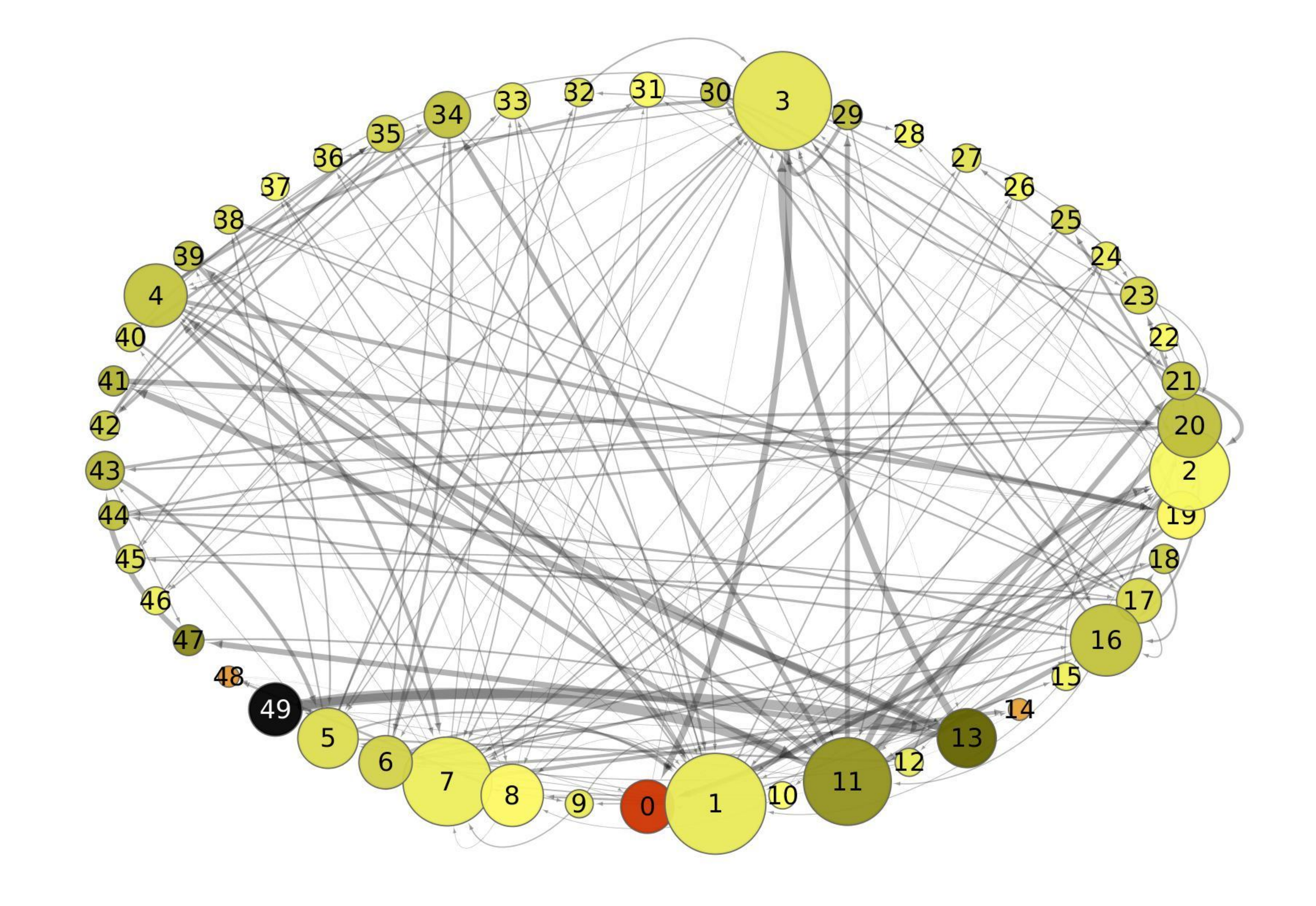}
    \caption{Example $3$. The structure of the scale-free network with $50$ nodes,
      as well as the information of the reactive trajectories from node $49$
      (black) to node $0$ (red) are displayed. The colors of the nodes depend on the committor value
    $q$ and brighter colors indicate that the values $q$ are larger. Nodes with
  relatively larger values $\widetilde{\theta}$ are shown in larger sizes.
Edges with relatively larger fluxes $\widetilde{J}$ are shown with
thicker lines. Also see Table~\ref{tab-ex3-1} for quantitative results. \label{fig-ex3-1}}
\end{figure}
\begin{table}[h]
\centering
  \begin{tabular}{cccccccccccc}
    \hline
    \hline
    Node & $49$ & $48$ & $14$ & $10$ & $22$ & $\cdots$ & $2$ & $11$ & $7$ & $3$ & $1$ \\
     Degree & $2$ & $2$ & $2$ & $2$ & $2$ & $\cdots$ & $9$ & $9$ & $11$ & $14$ & $15$ \\
    \hline
    \hline
    Node & $49$ & $13$ & $47$ & $11$ & $41$ & $\cdots$ & $8$ & $19$ & $14$ & $48$ & $0$ \\
     $q$ & $0.0$ & $0.51$ & $0.57$ & $0.59$ & $0.63$ & $\cdots$ & $0.74$ & $0.75$ & $0.86$ & $0.87$ & $1.0$ \\
    \hline
    \hline
    Node & $0$ & $48$ & $14$ & $10$ & $22$ & $\cdots$ & $7$ & $13$ & $3$ & $1$ & $11$ \\
    $\bar{\theta}'$ & $0.00$ & $0.03$ & $0.04$ & $0.13$ & $0.14$ & $\cdots$ &
			   $0.83$ & $1.09$ & $1.14$ & $1.18$ & $1.40$ \\
    \hline
    \hline
    Node & $48$ & $14$ & $10$ & $22$ & $37$ & $\cdots$ & $2$ & $11$ & $7$ & $3$ & $1$ \\
    $\widetilde{\theta}$ & $0.20$ & $0.22$ & $0.36$ & $0.36$ & $0.36$ &
			      $\cdots$ & $1.63$ & $1.99$ & $2.05$ & $2.67$ & $2.83$ \\
    \hline
    \hline
    Node & $48$ & $14$ & $10$ & $22$ & $37$ & $\cdots$ & $2$ & $7$ & $11$ & $3$ & $1$ \\
$\theta$ & $0.23$ & $0.26$ & $0.48$ & $0.50$ & $0.50$ & $\cdots$ & $2.24$ & $2.88$ & $3.39$ & $3.81$ & $4.01$ \\
    \hline
    \hline
    \end{tabular}
    \caption{Example $3$. Various statistics related to the first passage
    paths of the random work from node $49$ to node $0$.
    The first and the last five nodes of the node lists which are sorted according to
    values of each statistical quantity are shown. The network is displayed in
    Figure~\ref{fig-ex3-1}. Row with label ``Degree'' shows the degrees of the
  corresponding nodes. Definitions of other quantities can be found in
Section~\ref{sec-reactive} and Table~\ref{tab-notation}.\label{tab-ex3-1}}
\end{table}
\subsection{Example $4$ : football match}
\label{subsec-football}
In the last example, we apply the analysis in Section~\ref{sec-reactive} to
study the football match between Germany and Argentina in the $2014$ world cup
final~\cite{2014final}. The main aim is to quantify the performance of the players according to the ball passes during the match.
For simplicity, we only focus on the German national team and model the
ball passes between players as a discrete-time Markov jump process.

We take the data-based approach as discussed in Section~\ref{sec-algo}.
First, all passes between players of the Germany
national team during the entire game ($90$ minutes plus $30$ minutes extra
time) were recorded~\cite{youtube_videos}.
These passes constitute $107$ trajectories, each of which describes
consecutive ball passes among German players when the team possessed the ball.
Since each trajectory may start or finish in different ways, besides the $11$
starters and $2$ substitutes who participated in the game~(see Table~\ref{tab-ex4-1}), $6$ additional
nodes are introduced in the state set as outlined below.
Specifically, a trajectory starts from node $1$ if the goalkeeper
starts with a kick, and it starts from either node S$0$ or node S$1$ when the
Germany team got the possession from the opponent (or via a throw-in), depending on
whether it occurred in the defensive Half or attacking Half. Similarly, for the end
state, node L$0$ and L$1$ indicate that the German team lost possession of the
ball in the defensive Half and attacking Half, respectively.
Furthermore, both nodes B$0$ and B$1$ mark that the attack finished in the
opponent's penalty area, while the latter indicates a shot on goal was made.

We estimate the transition probabilities $p$ and the initial distribution $\mu$ of the graph using formulas
(\ref{kappa-p-estimator}), (\ref{init-mu-estimator}) in Section~\ref{sec-algo}. Besides the three initial nodes S$0$, S$1$ and $1$
(Neuer), we also select nodes $4$, $5$ and $20$ into the set $A$, since they
correspond to the three full-back players (H\"{o}wedes, Hummels, Boateng) who
stayed in the defensive Half most of the time during the match.
Therefore, we have sets $A=\{\mbox{S}0\,, \mbox{S}1\,, 1\,, 4\,, 5\,, 20\}$ and the initial distribution for nodes $\{$S$0$, S$1$, $1$,
$4$, $5$, $20$$\}$ in set $A$ is $\mu = \{0.579, 0.252, 0.168, 0, 0, 0\}$.
Set $B$ is chosen to contain all end nodes of the trajectories, i.e. $B =
\{\mbox{L}0\,, \mbox{L}1\,, \mbox{B}0\,, \mbox{B}1\}$. This choice of set $B$
allows us to utilize all the trajectories and the results can provide us an
overall insights of this specific match, i.e. not only shots on goal but also 
attacks which were not successful.

After these preparations, we apply Algorithm~\ref{algo-1} to compute the
various statistics defined in Section~\ref{sec-reactive} and results are
shown in Figure~\ref{fig-ex4-1} and Figure~\ref{fig-ex4-2}.
As already mentioned in Section~\ref{sec-algo}, function
$\theta$, flux $J$ and the average total length of the first passage paths computed
from Algorithm~\ref{algo-1} are identical to those computed from
estimators in (\ref{right-estimator}). In Figure~\ref{fig-ex4-1}, nodes and
edges are displayed in different sizes and thickness according to the values of
function $\theta$ and flux $J$. It could be observed that the four
players Schweinsteiger, Kroos, Lahm, \"{O}zil had much possession of the ball and contributed
more ball passes compared to the other players.
In Figure~\ref{fig-ex4-2}, on the other hand, the sizes of nodes and the thickness
of edges are determined according to their values of function
$\widetilde{\theta}$ and $\widetilde{J}$, respectively, which are related to the reactive
trajectory ensemble. We see that passes among the goalkeeper and the full-back players
(Neuer, H\"{o}wedes, Hummels, Boateng) are filtered out. Furthermore, upon closer
examination of the edges, we can observe that the role of the attackers
(e.g. \"{O}zil, M\"{u}ller) become more prominent. Quantitative results
are presented in Table~\ref{tab-ex4-2}.

\begin{table}[htpb]
  \centering
  \begin{tabular}{ll|ll}
    \hline
    \hline
    No. & Name & No. & Name \\
    \hline
    $1$ & Neuer (GK) & $16$ & Lahm (C) \\
    $4$ & H\"{o}wedes &  $18$ & Kroos\\
    $5$ & Hummels & $20$ & Boateng\\
    $7$ & Schweinsteiger & $23$ & Kramer ($\downarrow 31'$) \\
    $8$ & \"{O}zil ($\downarrow 120'$) & $9$ & Sch\"{u}rrle ($\uparrow 31'$)\\
    $11$ & Klose ($\downarrow 88'$) & $17$ & Mertesacker ($\uparrow 120'$)\\
    $13$ & M\"{u}ller &  $19$ & G\"{o}tze ($\uparrow 88'$)  \\
    \hline
  \end{tabular}
  \caption{German national team players in the FIFA world cup $2014$ final.
    Information of the substitutions are indicated in the brackets. Notice that
    player No.$17$ (Mertesacker) does not appear in the graph because he did
    not contribute to the passes in the trajectories.
  \label{tab-ex4-1}}
\end{table}
\begin{table}[htbp]
  \centering
  \begin{tabular}{c|c|c|c|c|c|c|c|c|c}
    \hline
    \hline
    Node & $q$ & $\bar{\theta}'$ & $\widetilde{\theta}$ & $\theta$ & Node & $q$ & $\bar{\theta}'$ & $\widetilde{\theta}$ & $\theta$\\
    \hline
    $1$ & $0.00$ & $0.25$ & $0.10$ & $0.35$ & $16$ & $0.40$ & $0.64$ & $0.43$ & $1.07$ \\
    $4$ & $0.00$ & $0.34$ & $0.29$ & $0.63$ & $18$ & $0.41$ & $0.62$ & $0.43$ & $1.05$ \\
    $5$ & $0.00$ & $0.39$ & $0.14$ & $0.53$ & $20$ & $0.00$ & $0.46$ & $0.27$ & $0.73$ \\
    $7$ & $0.36$ & $0.63$ & $0.36$ & $0.99$ & $23$ & $0.48$ & $0.05$ & $0.05$ & $0.09$ \\
    $8$ & $0.47$ & $0.41$ & $0.36$ & $0.77$ & $9$ & $0.45$ & $0.34$ & $0.28$ & $0.62$ \\
    $11$ & $0.68$ & $0.07$ & $0.15$ & $0.22$ & $19$ & $0.47$ & $0.08$ & $0.08$ & $0.16$ \\
    $13$ & $0.55$ & $0.28$ & $0.35$ & $0.63$ & B$0$ & $1.00$ & $0.00$ & $0.24$ & $0.24$ \\
		       & & & & & B$1$ & $1.00$ & $0.00$ & $0.06$ & $0.06$ \\
    \hline
  \end{tabular}
  \caption{Example $4$. Various statistics related to the first passage paths
  in the football match. Definitions of these quantities can be found in
Section~\ref{sec-reactive} and Table~\ref{tab-notation}. See Figure~\ref{fig-ex4-1} and Figure~\ref{fig-ex4-2} for a depiction of the network. \label{tab-ex4-2}}
\end{table}
\begin{figure}[hptb]
\centering
\includegraphics[width=14cm]{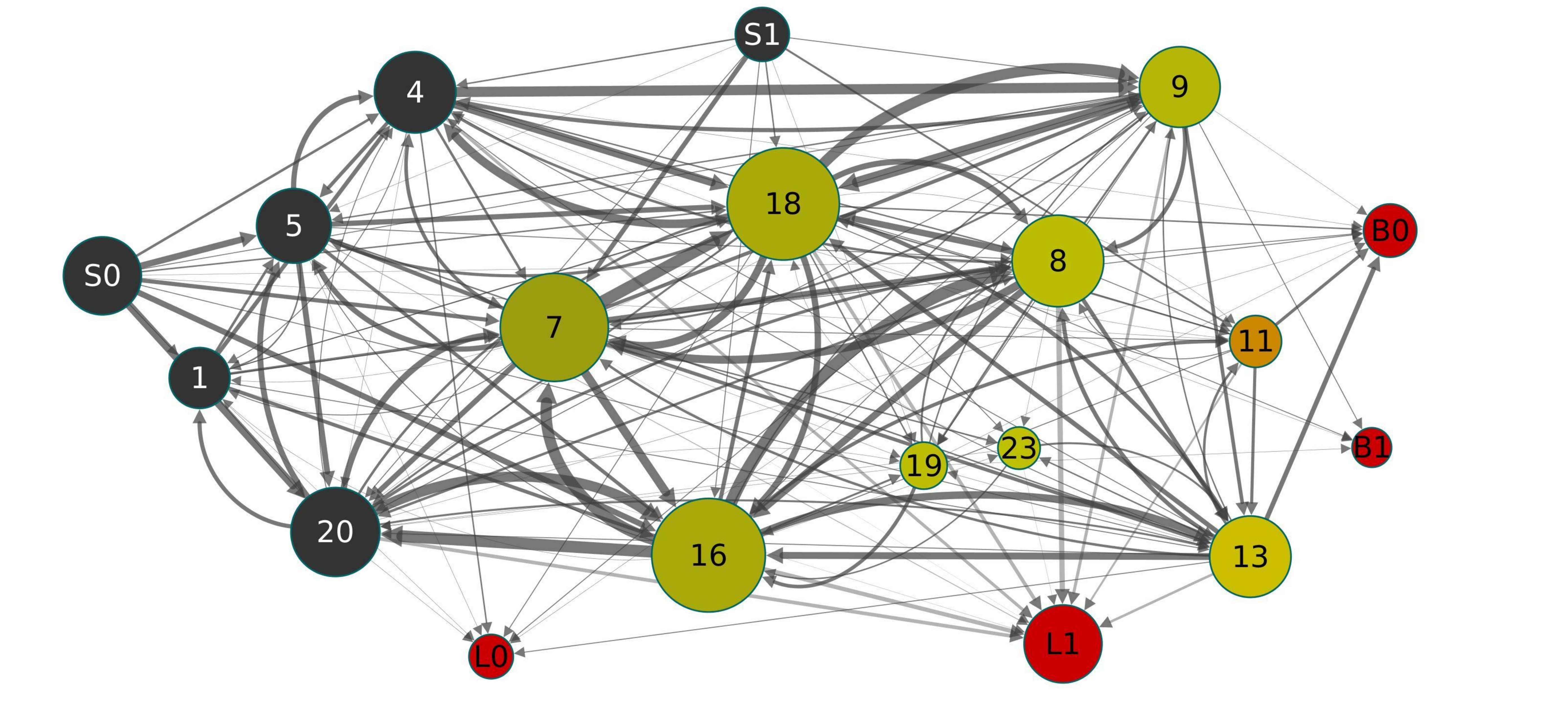}
    \caption{Example $4$. The graph modelling of the football match and the
    information related to the first passage path ensemble. Nodes S$0$, S$1$
  and L$0$, L$1$, B$0$, B$1$ are introduced to mark the start and end status of the
first passage path (consecutive passes). The other nodes with digital labels correspond to
the players listed in Table~\ref{tab-ex4-1}. Nodes in set $A$ and $B$ are
plotted in black and red, respectively. Nodes with larger values of $\theta$ 
(number of times the node is visited on first passage path) are
shown in larger sizes. Edges with relatively larger values of fluxes $J$ are
shown in thicker lines. Colors indicate the values of the committor function
([low] black-green-red [high]).
\label{fig-ex4-1}}
\end{figure}
\begin{figure}[hptb]
\centering
\includegraphics[width=14cm]{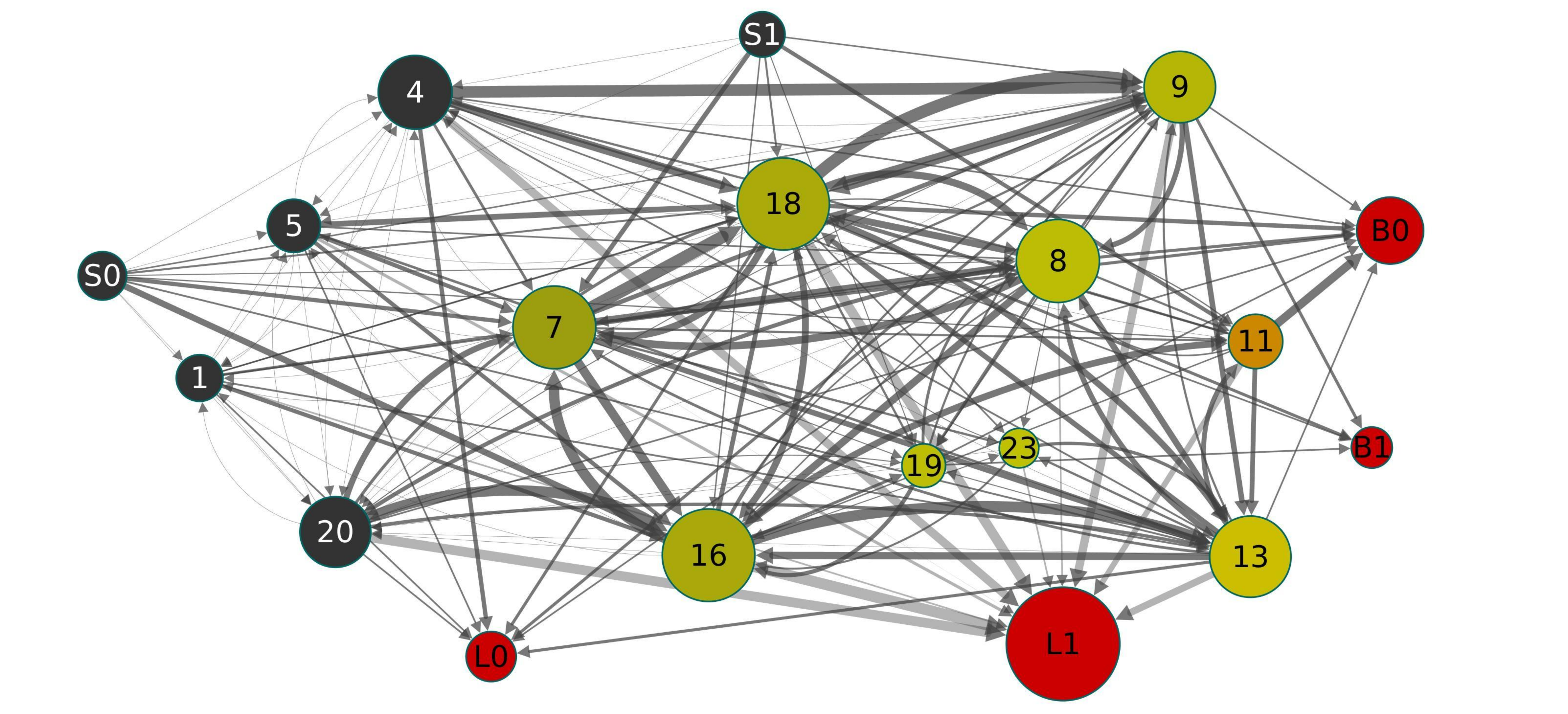}
\caption{Example $4$. The graph modelling of the football match and the
    information related to the reactive path ensemble. Different from
    Figure~\ref{fig-ex4-1}, in this figure, nodes with larger values of
    $\widetilde{\theta}$ (number of times that the node is visited by a reactive trajectory)
    are
shown in larger sizes. Edges with relatively larger values of reactive fluxes
$\widetilde{J}$ are shown in thicker lines. 
Colors indicate the values of the committor function ([low] black-green-red [high]). \label{fig-ex4-2}}
\end{figure}
\section{Conclusions}
\label{sec-conclusion}
In the present work, we have developed a theory for the statistics of the first
passage path ensemble of jump processes on a finite state space. The main
approach is to divide a first passage path into nonreactive and reactive segments
so that the behaviors of the processes within each segment can be studied separately.
Furthermore, relations between the statistics of these two segments
and the statistics of the entire first passage paths are
obtained.

Our analysis can be applied to jump processes which are non-ergodic, as well as continuous-time jump processes where the waiting time
distributions are non-exponential. More generally, second-order or higher
order Markov chains have been studied where the dynamics
depend not only on system's current state but also on the states which have
been visited in the previous
steps~\cite{Rosvall2014,Scholtes2014,mem_dectect2014}. These types of jump
processes can be converted to Markov jump processes by extending the state space and therefore it is possible
to apply the analysis in the current work to study these higher order Markov
chains (although the size of the extended state space may become quite large and brings computational difficulties). 

We expect that the study of both the nonreactive and reactive segments can
help to understand the transition behavior of processes from one subset to
another. Especially, analysis of the reactive segments in this work is closely related
to TPT, which was developed for both diffusion processes and jump processes under the
assumption of ergodicity. Some illustrative examples are studied numerically in order to demonstrate the
applicability of the theory. More generally, the study of the transition events, or the first passage
phenomena~\cite{metzler2014first},
plays an important role in order to understand many real-world systems and applications, e.g. in molecular dynamics or in epidemiology.
Applying the analysis of the current paper to study phenomena in the aforementioned subject areas will be considered in future work.

\section*{Acknowledgement}
This research has been funded by Deutsche
Forschungsgemeinschaft (DFG) through grant CRC 1114, by the Einstein
Foundation Berlin through project CH4 of Einstein Center for Mathematics
(ECMath) and though the BMBF, grant number 031A307.  
\appendix
\section{An alternative expression of $\mu_r$}
\label{appsec-1}
    Here we provide an alternative way of studying the probability
    distribution $\mu_r$ defined in (\ref{mu-r}) in Section~\ref{sec-reactive}. Consider the
    first passage paths starting from $x \in
	B^c$, and for each $y \in A$, let $\omega(x,y)$ be the probability that $y$ is the
	last hitting node in $A$, i.e.
	$\omega(x,y) = \mathbb{P}(x_{\sigma_x}=y)$. The equation of $\omega$ can be
	obtained by considering the next state $z$ which the system will jump to from $x$.
	In the case $x \neq y$, we obtain the relation
	\begin{align}
	  \omega(x,y) =& \sum_{z \in B^c} p(z\,|\,x)\,\omega(z,y), \quad x \neq
	  y\,, y \in A\,.
	  \label{pi-1}
	\end{align}
	Notice that for $x \not\in V^-$, we have $\omega(x,y) = 0$ and (\ref{pi-1}) still holds.

	In the case $x=y \in A$, after jumping to another state $z$,
	there are two possibilities depending on whether the system will
	return to set $A$ again before it reaches set $B$. Notice that such
	events are characterized exactly by $\{\tau_{A,z} < \tau_{B,z}\}$ and its
	complement. Using the committor function $q$ in
	(\ref{committor-q}), (\ref{committor}), we obtain
	\begin{align}
	  \omega(y,y) =& \sum_{z \in B^c} p(z\,|\,y)\,\omega(z,y) + \sum_{z \in V}
	p(z\,|\,y) q(z)\,.
	  \label{pi-2}
	\end{align}
	For each $y \in A$, we can solve $\omega(\cdot,y)$ from (\ref{pi-1})-(\ref{pi-2}).
	Since the initial distribution is $\mu$ in the first passage path ensemble, we
	obtain
	\begin{align}
	\mu_r(x) = \sum\limits_{y \in A} \mu(y) \omega(y,x), \quad x \in A\,.
      \end{align}
      \section{Path probabilities for a continuous-time jump process}
      \label{appsec-2}
In the following, starting from a continuous-time jump process, we study the
probability associated to its jump trajectories and
further clarify the connection between the continuous-time and discrete-time
(Markov) jump process defined by transition probabilities $p$ in
Subsection~\ref{subsec-continuous-to-discrete}.

Following the setup in Subsection~\ref{subsec-continuous-to-discrete}, let $\rho(t,x)$ and $w(t,x)$ denote the probabilities that the system stays at node $x$ at time $t$
and the system jumps to node $x$ (from another node) at time $t$,
respectively. They are related by
    \begin{align}
      \rho(t,x) = \int_{0}^t w(t-u,x) a(u\,|\,x)\, du\,,
      \label{eqn-p}
    \end{align}
  and we have $w(0,x) = \rho(0,x)$ at time $t = 0$. Considering the node $y$ from
  which the system jumps to $x$, we can obtain the equation
    \begin{align}
      w(t,x) = \sum_{y \in V,\, y\rightarrow x} \int_0^t w(t-s,y)
      b(s\,,y\rightarrow x) ds + w(0,x) \delta(t) \,.
      \label{eqn-q}
    \end{align}
    Expressions of $\rho(t,x), w(t,x)$ have been obtained in
    \cite{non_poisson_jump} by applying
    Laplace transformation in (\ref{eqn-p}) and (\ref{eqn-q}).

    A trajectory of the continuous-time jump process can be represented as a path with
    time-stamps
    \begin{align}
    \varphi=\big\{(t_0, x_0), (t_1, x_1), \cdots, (t_N, x_N),\,\cdots \big\},
    \label{path-with-time}
  \end{align}
    where $t_0=0$ and $0\le t_1 < \cdots < t_N$ are the jump time,
    $x_k \in V$, $0 \le k \le N$. It corresponds to the event that the
    process starts from state $x_0$ at time $t_0=0$ and later on jumps
    consecutively from node $x_k$ to $x_{k+1}$ at time $t_{k+1}$, $0 \le k \le N-1$.
    The probability density that this specific path $\varphi$ occurs is
    \begin{align}
      \rho(0,x_0) \prod_{k=0}^{N-1} b\big(t_{k+1} - t_k\,,x_k\rightarrow
      x_{k+1}\big)\,.
    \end{align}
    Now we consider the path $\varphi$ when the sequence of nodes $(x_0, x_1,
    \cdots, x_N)$ is fixed and the jump times are allowed to vary.
 For $1 \le k \le N$, we denote the probability
    density that the system arrives at node $x_k$
    at time $t$ along the sequence of nodes $(x_0, x_1,
    \cdots, x_k)$ by $r_k(t)$. Also, let $\eta(t)$ be the probability that the system arrives
    at state $x_N$ along the path $(x_0, x_1, \cdots, x_N)$ and remains there
    until time $t$. Notice that we have omitted the dependence on this specific node sequence in the
    notations of $r_k(\cdot)$ and $\eta(\cdot)$ for simplicity.
    Setting $b_k(\cdot)=b(\cdot\,,x_k \rightarrow x_{k+1})$ for $0 \le k < N$,
    then clearly we have $r_1(t) = b_0(t)$ and
    \begin{align}
      r_k(t) = \int_0^t du_1 \int_{u_1}^t du_2 \cdots
      \int_{u_{k-2}}^t
	du_{k-1}
	\Big[\prod_{l=0}^{k-2} b_{l}(u_{l+1} - u_l)\Big]
	b_{k-1}(t-u_{k-1})\,,
	\label{r-k-formula}
    \end{align}
    for $2 \le k \le N$, and also $\eta(t) = \int_{0}^t r_N(u)\, a(t-u\,|\,x_{N})\, du$.
    From (\ref{r-k-formula}), it is direct to verify that
    \begin{align}
    r_k(t) = \int_0^t r_{k-1}(u) b_{k-1}(t-u) du\,, \quad  2 \le k \le N\,,
    \label{r-k-recursive}
  \end{align}
  and therefore we can obtain
    \begin{align}
      \begin{split}
	\hat{r}_k(s) =& \hat{r}_{k-1}(s)\, \hat{b}_{k-1}(s)\,, \quad 2 \le k \le
      N\,, \\
      \hat{\eta}(s) =& \hat{r}_N(s)\, \hat{a}(s\,|\,x_N)\,,
    \end{split}
      \label{laplace-rel}
    \end{align}
    where $\hat{r}_k$, $\hat{\eta}$, $\hat{b}_k$, $\hat{a}$ denote the Laplace transformations of
    functions $r_k, \eta$, $b_k$ and $a(\cdot\,|\,x_{N})$, respectively.
    As a result, we obtain the expressions
    \begin{align}
	\hat{r}_k(s) = \prod\limits_{l=0}^{k-1} \hat{b}_l(s) \,,
	\quad
	\hat{\eta}(s) = \hat{a}(s\,|\,x_N)\prod\limits_{l=0}^{N-1} \hat{b}_l(s) \,,
	\label{r-eta-formula}
    \end{align}
    where $1 \le k \le N$.
  When the probability densities $\psi$ of the waiting times are exponential distributions, (\ref{r-eta-formula}) can be further simplified and the expression of
	$\eta(t)$ has been obtained in \cite{path_summation}. Also see
	\cite{path_sampling_sun} for a sampling algorithm based on these
	quantities.

    From (\ref{r-k-formula}) and (\ref{p-b-formula}), we can also compute the
    probability that the system arrives at $x_N$ along
    the specific path $x_0, x_1, \cdots, x_N$ (regardless of the jump times) and obtain
    \begin{align}
      \rho(0,x_0) \int_{0}^{+\infty} r_N(t)\,dt = \rho(0,x_0) \prod_{k=0}^{N-1}
      p(x_{k+1}\,|\, x_k)\,.
    \end{align}
    The right hand side of the expression above indicates that
    we can focus on the discrete-time Markov jump process on $G$ defined by the transition
    probability matrix $P$, whose entries are $P_{xy} = p(y\,|\,x)$, if we are
    interested in the jump paths irrespective of the time information.

\bibliographystyle{siam}
\bibliography{reference}
\end{document}